\documentclass[11pt,a4paper,final]{article}

\RequirePackage[backend=biber,style=authoryear,natbib,
  uniquename=false, maxbibnames=4]{biblatex}
\RequirePackage{csquotes}
\usepackage{authblk}
\usepackage{bm}
\usepackage{amsmath}    %
\usepackage{amsfonts}   %
\usepackage{amssymb}    %
\usepackage{amsthm}     %
\usepackage{mathrsfs}   %
\usepackage{parskip}    %
\usepackage{calc}       %
\usepackage{enumitem}   %
\usepackage{graphicx}   %
\usepackage{lpic}       %
\usepackage{verbatim} %
\usepackage[dvipsnames]{xcolor}
\usepackage{tcolorbox}
\usepackage[english]{babel}
\usepackage{breakurl}
\RequirePackage{hyperref}
\hypersetup{
    colorlinks=true, %
    linktoc=all,     %
    linkcolor=blue,  %
    citecolor=blue,
    urlcolor = blue,
    }
\usepackage{cleveref}
\crefformat{assumption}{assumption #2#1#3}
\crefrangeformat{assumption}{assumptions~#3#1#4 to~#5#2#6}
\crefmultiformat{assumption}{assumptions~#2#1#3}{ and~#2#1#3}{, #2#1#3}{ and~#2#1#3}
\crefrangemultiformat{assumption}{assumptions~#3#1#4 to #5#2#6}{ and #3#1#4 to #5#2#6}{, #3#1#4 to #5#2#6}{ and #3#1#4 to #5#2#6}

\newlength{\vslength}
\newcommand{\eg}{{\it e.g.}}

\newcommand{\etc}{{\it etcetera}}

\newcommand{\RR}{{\mathbb R}}

\newcommand{\scrB}{{\mathscr B}}

\newcommand{\scrG}{{\mathscr G}}

\newcommand{\scrP}{{\mathscr P}}

\newcommand{\scrX}{{\mathscr X}}

\DeclareMathOperator{\sL}{\mathcal{L}}

\DeclareMathOperator{\sO}{\mathcal{O}}
\DeclareMathOperator{\sP}{\mathcal{P}}

\newcommand{\samplen}{{X^{n}}}

\newcommand{\realizationn}{{x^{n}}}

\newcommand{\eps}{{\varepsilon}}

\newcommand{\tht}{{\theta}}

\newcommand{\Tht}{{\Theta}}

\newcommand{\set}[1]{\left\{ #1 \right\}}

\newcommand{\ft}[2]{{\textstyle{\frac{#1}{#2}}}}

\newcommand{\conv}[1]%
  {{\mathrel{\,\xrightarrow{\widthof{\,#1\,}}\,}}}
\newcommand{\convas}[1]%
  {{\mathrel{\,\xrightarrow{\widthof{\,#1\text{-a.s.}\,}}\,}}}
\newcommand{\convprob}[1]%
  {{\mathrel{\,\xrightarrow{\widthof{\,#1\,}}\,}}}
\newcommand{\convweak}[1]%
  {{\mathrel{\,\xrightarrow{\widthof{\,#1\text{-w.}\,}}\,}}}

\newcommand{\twobytwo}[4]%
  {\left(\begin{array}{cc} #1 & #2 \\ #3 & #4 \end{array}\right)}
\newcommand{\twovec}[2]%
  {\left({\begin{array}{c} #1\\#2 \end{array}}\right)}

\newcommand{\ceiling}[1]{\left\lceil #1 \right\rceil}

\newcommand{\floor}[1]{\left\lfloor #1 \right\rfloor}

\newcommand{\rh}[1]{\left(#1\right)}
\newcommand{\vh}[1]{\left[#1\right]}
\DeclareMathOperator{\en}{\quad \text{and}\quad}
\DeclareMathOperator{\weg}{\backslash}
\renewcommand{\qedsymbol}{$\Box$}
\newcommand{\closebox}{\hfill\qedsymbol}

\newcommand{\expa}[1]{\exp\left(#1\right)}

\newcommand{\weghalen}[1]{}

\newtheoremstyle{customtheorem}%
  {0.5em}%
  {0.2em}%
  {\itshape}%
  {}%
  {\scshape}%
  {}%
  {1ex}%
  {}%

\theoremstyle{customtheorem}
\newtheorem{theorem}{Theorem}[section]
\newtheorem{lemma}[theorem]{Lemma}
\newtheorem{proposition}[theorem]{Proposition}
\newtheorem{corollary}[theorem]{Corollary}
\newtheorem{definition}[theorem]{Definition}
\newtheorem{assumption}[theorem]{Assumption}
\newtheoremstyle{customremark}%
  {0.5em}%
  {0.2em}%
  {}%
  {}%
  {\scshape}%
  {}%
  {1ex}%
  {}%

\theoremstyle{customremark}
\renewenvironment{proof}{\par\noindent{\scshape Proof}\;}{\hfill\qedsymbol\par}
\newtheorem{remark}[theorem]{Remark}
\newtheorem{example}[theorem]{Example}

\newcommand{\extra}[1]{{}}

\title{%
Uncertainty quantification in the stochastic block model with an unknown number of classes}
\author[1]{J. van Waaij\thanks{The first author is supported by the University of Padova under the STARS Grant.
}}
\author[2]{B.J.K. Kleijn}
\affil[1]{Department of Statistics,
	University of Padova}
\affil[2]{Korteweg-de~Vries Institute for Mathematics,
	University of Amsterdam}
\affil[1]{\href{mailto:jvanwaaij@gmail.com}{jvanwaaij@gmail.com}}
\affil[2]{\href{mailto:B.J.K.Kleijn@uva.nl}{b.j.k.kleijn@uva.nl}}
\date{\today}
\begin{document}
	\sloppy
\maketitle

\begin{abstract}\noindent
We study the frequentist properties of Bayesian statistical inference for the stochastic block model, with an unknown number of classes of varying sizes. We equip the space of vertex labellings with a prior on the number of classes and, conditionally, a prior on the labels. The number of classes may grow to infinity as a function  of the number of vertices, depending on the sparsity of the graph.  
We derive  non-asymptotic posterior contraction rates of the form \(P_{\theta_{0,n}}\Pi_n(B_n\mid X^n)\le \eps_n\), where \(X^n \) is the observed graph, generated according to \( P_{\theta_{0,n}}\),  \(B_n\) is either \(\set{\theta_{0, n}}\) or, in the very sparse case, a ball around \(\theta_{0,n}\) of known extent, and \(\eps_n\) is an explicit rate of convergence. 

  These results enable conversion of credible sets to confidence sets. In the sparse case, credible tests are shown to be confidence sets. In the very sparse case, credible sets are enlarged to form confidence sets. Confidence levels are explicit, for each \(n\), as a function of the credible level and the rate of convergence.
  Hypothesis testing between the number of classes is considered with the help of posterior odds,  and is shown to be consistent. Explicit upper bounds on errors of the first and second type and an explicit lower bound on the power of the tests are given.
\end{abstract}

\section{Communities in random graphs}
\label{sec:intro}

Networks are more present than ever in history. The emergence of the Internet, which is only about 30 years old, is clearly one of the most eye-catching examples, but also developments in biology giving rise to enormous  networks, all waiting for the  statistician to be analysed.

Under a network or graph we understand a collection of vertices (also called nodes) and edges between vertices, which can be directed or undirected. Very often data is associated with the vertices and edges. For example in the Facebook network, the vertices are users and edges are formed when two users are friends. With the vertex the user's name, age, etc. are associated,  and with the edges for example the date that the connected users became friends could be stored.

There are many things that can be 
measured from a network, however, 
the most intrinsic to a graph is its 
geometric form, emerging from the 
connections between the vertices. One 
interesting question that may arise is, are there communities in a 
graph? Communities are generally understood to be sets of 
vertices that are more densily connected among each other, than to vertices outside the set. 
Communities are an indication of shared common properties, for example in the yeast protein-protein interaction network, communities (so-called ``functional modules'') are ``cellular entities that perform certain biological functions, which are relative independent of each other'' (\cite{chenyuan2006}).

We use the stochastic block model (SBM), as a mathematical model for communities in a network. 
The SBM was first considered in \cite{Holland83}. In the SBM we observe a 
graph $X^n$ with \(n\) vertices, %
in
which vertices belong to one of a finite number of classes, while edges
occur independently, with probabilities that depend on the classes
of the vertices they connect. 

SBMs have many applications in science and machine learning applications. One interesting example is  modelling gene expression (\cite{Clineea2007}). Another example is the already mentioned yeast protein-protein interaction network studied in \cite{chenyuan2006}. More examples and references of applications are provided in the first section of \cite{Abbe18}. 

There is now a substantial body of 
literature on the SBM that deals with 
recovery of communities in a 
network. %
Transition phases for the SBM for different number of blocks are studied in \cite{abbesandon2015}, \cite{abbesandon2018}, and \cite{zhangzhou2016}.
Rates of posterior convergence for SBMs are examined in \cite{mariadassoumatias2015}, \cite{pasvaart2018}, \cite{ghoshea2019},  and  \cite{gengbhattacharyapati2019}. A Bayesian framework for estimating the parameters of a SBM (number of classes, connecting probabilities) is considered in \cite{gaovaartzhou2015}.  For a comprehensive overview one consults \cite{Abbe18}. 

 Recent work mainly considers the SBM under the assumption that the number of blocks is known, or unknown but bounded. For example \cite{zhangzhou2016}, allows the number of blocks to be as large as \(\sO(n/\log n)\), but still assumes it to be known.  \cite{gengbhattacharyapati2019} allows the unknown number to be two or three, but state that the general problem (four or more classes) is unsolved.
 
 In the non-Bayesian approach this is less a problem, as one can first estimate the parameters of the SBM (using  for instance one of the algorithms mentioned in \cite[sec. 7]{Abbe18}), and then recover the clustering. From a Bayesian perspective this is unnatural (unless one uses an empirical Bayes approach). Here we give a partial answer to this problem. We allow the number of blocks to be unknown and even as large as \(n^\alpha,\) for some \(0<\alpha<1/2\) (in the dense phase, see \cref{ex:selectdense}).

 This also gives a partial answer to the open problem in \cite[sec. 8]{Abbe18}, to whether exact recovery is possible in classes of sub-linear size, in this case of size \(n^{1-\alpha}\). Another question that he raises (p. 4) and tries to find an answer to is:
 \begin{quote}
 	Are there really communities? Algorithms may output community structures, but are these meaningful or artefacts?
 \end{quote}
We do this here in several ways. One is by providing confidence sets for our estimators, and another is by providing frequentist guarantees for testing with posterior odds, whether the data comes from a Erd\H os-Ren\`\i\ graph model (there are no communities), or a SBM (there are communities). 

\paragraph{Outline of our results}

We study posterior rates of convergence for large classes of priors. Posterior rates of convergence are balls \(B_n\) of labellings centred around the true labelling \(\theta_{ 0,n}\) of a certain radius \(r_n\) (in a metric to be specified) and a sequence \((\eps_n)\) such that \[
P_{\theta_{ 0,n}}\Pi^n (B_n\mid X_n )\ge 1-\eps_n. 
\]
For graphs that are not too sparse, we take \(B_n = \set{\theta_{ 0,n}}\), which leads to exact detection. 

This unusual precise description of 
the posterior convergence rates 
with \(\eps_n\) and \(r_n\) allows us 
to construct confidence sets from 
the credible sets and to derive explicit bounds on the errors for testing with posterior odds. 

In the  sparse and dense cases, credible sets 
are shown to be confidence sets 
with an exact confidence level, which is a function of the credible level and the posterior rate of convergence. In the very sparse 
case, credible 
sets are enlarged to convert them into confidence sets. 

For symmetric testing with posterior odds, we show that posterior convergence in a parameter set \(A_n\) allows us to consistently test between \(A_n\) and another disjoint set \(B_n\). This enables us for instance to test for the number of classes and in particular to test between the SBM and the Erd\H os-R\'enyi graph model (basically the SBM with one class).

In the next section we describe our model. In \cref{sec:SBMselect,sec:posteriorconsistencyataparameter} we derive posterior convergence and apply that to some interesting examples. Confidence sets are examined in \cref{sec:examplecrediblesets} and symmetric testing in \cref{sec:hypothesistesting}. Proofs are mostly defered to  \cref{app:PBMtests,sec:aux}. A brief review of the most relevant Bayesian theory can found in \cref{app:defs}. Notation and conventions are also explained in this appendix.

\section{The stochastic block model}
\label{sec:sbm}

In a SBM for a random graph $X^n$ of order $n$,
each vertex is assigned to a class through an
unobserved \emph{class assignment vector} $\theta'_n$.
The space in which the random graph $X^n$ takes its values
is denoted $\scrX_n$ (\eg\ represented by its (random) adjacency matrix
with entries $\{X_{ij}:1\leq i<j\leq n\}$).
Each vertex belongs to a class and any edge occurs (independently
of others) with a probability depending on the classes
of the vertices that it connects. 
We study 
the \emph{planted multi-section model}, in which the vertices
are divided into $\ell\geq1$ groups of $1\le m_{n,1,1}\le \ldots \le m_{n,\ell,\ell}$ vertices,
for a total of $n=m_{n,\ell,1}+\ldots+m_{n,\ell,\ell}$ vertices (with labels $1,\ldots,\ell$). The corresponding \(\ell\)-vector is denoted by \(\bm m_{n,\ell}=(m_{n,\ell,1},\ldots,m_{n,\ell,\ell})\). 
We define the set $\Theta'_{\bm m_{n,\ell}}$ to be the subset of 
$\{1,\ldots,\ell\}^{n}$ consisting of all class assignment vectors
$\theta'=(\theta_{1}',\ldots,\theta_{n}'),\) with labels \( \theta_i' 
\in\{1,\ldots,\ell\}, i=1,\ldots, n$, that have exactly  \(m_{n,\ell,1}\) elements of one label, and  \(m_{n,\ell,2}\) elements of another label, \etc.

In the SBM for a graph $X^n$ with
class assignment vector $\theta'_n\in\Theta_n'$,  the probability
of an edge occurring between vertices $1\leq i,j\leq n, i\neq j$ is
denoted $Q_{i,j;n}$ and depends on $n$, $\theta'_{n,i}$
and $\theta'_{n,j}$ only. If edge probabilities vary (that is,
if edge probabilities within different classes vary
or if, for example, the edge probability between classes one
and two is different from that between three and four) vertices
are classified correctly by analysis of their asymptotic degrees
\citep{channarond2012}. In the planted multi-section
model it is assumed that the probability that an edge occurs,
depends only on whether the vertices it connects belong to
the same class or to different classes: the probability
of an edge \emph{within the same class} is denoted $p_n\in(0,1)$;
the probability of an edge \emph{between classes} is denoted
$q_n\in(0,1)$,
\begin{equation}
  \label{eq:pms}
  Q_{i,j;n}(\theta')=\begin{cases}
    \,\,p_n,&\quad\text{if $\theta_{i}'=\theta_{j}'$,}\\
    \,\,q_n,&\quad\text{if $\theta_{i}'\neq\theta_{j}'$.}
  \end{cases}
\end{equation}

Note that if $p_n=q_n$, $X^n$ is a Erd\H os-R\'enyi graph and the class assignment $\theta'_n\in\Theta'_n$ is not
identifiable. Indeed asymptotic proximity to the Erd\H os-R\'enyi
graph (in the sense that the difference between $p_n$ and $q_n$ decreases too fast to
zero as $n\to\infty$) can make consistent community
detection impossible, see \cite{Mossel16} for example in the case of two blocks. 
We distinguish between the following three cases:\begin{description}
	\item[Dense phase:] The average degree is proportional to \(n\), so \(\liminf_{n\to\infty}p_n>0\) and \(\liminf_{n\to\infty}q_n>0\).\item[Chernoff-Hellinger phase:] The average degrees grows logarithmically. So \(p_n=\frac{a_n\log n}n\) and \(q_n=\frac{b_n\log n}n\) for certain bounded positive sequences so that \(\liminf_{n\to\infty}a_n>0\) and \(\liminf_{n\to\infty} b_n>0\).
	\item[Kesten-Stigum phase:] The average degree is bounded from below, so \(p_n=c_n/n\) and \(q_n=d_n/n\), for bounded sequences of positive numbers so that \(\liminf_{n\to\infty}c_n>0\) and \(\liminf_{n\to\infty} d_n>0\).
\end{description}
The ability to detect the correct labelling, depends of course on the differences between \(p_n, q_n,\) and  \(a_n, b_n\) and \(c_n, d_n\), respectively. Exact tresholds for these are studied in \cite{abbesandon2015}, and \cite{abbesandon2018}.

Identifiability is also lost due to invariance of the model
under permutation of class labels: for example with $\ell=3$, the class assignment vectors $\theta'_1=(112233)$,
$\theta'_2=(223311)$ and $\theta'_3=(331122)$ all
give rise to the same distribution for $X^6$. This is expressed via an
equivalence relation $\sim$ on $\Tht'_{n,\ell}.$ Two labels \(\theta_n',\theta_n''\) are said to be equivalent, if there is a permutation \(\pi\) of the labels \(\set{1,\ldots, \ell}\) such that for all $1\leq i\leq n$,
\[
  \theta'_{1,i}=\theta''_{2,\pi(i)}.
\]

We define for an \(\ell\)-vector of class sizes \(\bm m_{n,\ell}=(m_{n,\ell,1},\ldots, m_{n,\ell,\ell})\) satisfying, \(1\le m_{n,\ell,1}\le \ldots\le m_{n,\ell,\ell}\) and \(\sum_{i=1}^\ell m_{n,\ell,i}=n\), the model $\Tht_{\bm m_{n,\ell}}=\Tht'_{\bm m _{n,\ell}}/\sim$, which can be identified with the space of
all partitions of $n$ elements into $\ell$ sets of sizes $m_{n,\ell, 1},\ldots, m_{n,\ell,\ell}$.
One can see that \cref{eq:pms} does not depend on the chosen representation \(\theta'\) of \(\theta,\) hence we may define \begin{equation}\label{eq:pmsundepofrepr}
Q_{i,j,n}(\theta)=Q_{i,j,n}(\theta').
\end{equation} 

In our set-up, we have to put some constraints on the \(\bm m_{n,\ell }\). 
Hence, for given \(\ell\), we let \(M_{n,\ell}\) be the set of all allowed \(\ell\)-vectors \(\bm m_{n,\ell}\). By \(\sL_n\) we denote all \(\ell\) for which \(M_{n,\ell}\) is not empty.
We define \(M_n=\bigcup_{\ell=1}^n M_{n,\ell}\).
  We set \[
\Theta_{n,\ell}=\bigcup_{\bm m_{n,\ell}\in M_{n,\ell}}\Theta_{\bm  m_{n ,\ell}}.
\]
The full parameter space is denoted \[
\Theta_n=\bigcup_{\bm m_{n}\in M_n}\Theta_{\bm m_{n}}=\bigcup_{\ell\in \sL_n}\Theta_{n,\ell}. 
\]

For \(\ell\in \sL_n\), we define 
 \[m_{n,\ell,\min}= \min_{\bm m _{n,\ell}\in M_{n,\ell}} m_{n,\ell,1},\quad\text{and}\quad m_{n,\ell,\max}= \max_{\bm m _{n,\ell}\in M_{n,\ell}}m_{n,\ell,\ell}.\]
  (Recall that \(m_{n,\ell,1}=\min_{1\le i\le \ell}m_{n,\ell,i}\) and \(m_{n,\ell,\ell}=\max_{1\le i\le \ell}m_{n,\ell,i}\).)
 
It is noted explicitly that \(\ell=1\) and \(\bm m_{n,1}=(n)\) is also allowed, which allows us to test between the Erd\H os-R\'enyi graph model and a SBM with at least two classes. In case \(\ell=1\), \(\Theta_{\bm  m_{n,1 }}'\) consist of only one element: the \(n\)-vector \((1,\ldots, 1)\).

In order to distinguish between the number of classes, in our approach, the minimum and maximum sizes of the classes
are required to adhere to the following assumption:
\begin{assumption}\label{ass:assumptiononthesizeoflabelsets}
For all \(\ell_1,\ell_2\in\sL_n\), whenever \(\ell_1< \ell_2\),   \(m_{n,\ell_1,\min}\ge m_{n,\ell_2,\max}\).  
\end{assumption}
This assumption organises class sizes for decreasing number of classes in consecutive intervals, concentrated around the mean number of elements 
 \(m:=n/\ell\) per class. For instance, if we impose an upper bound \(L_n\) for \(\ell\) given \(n\), then any vector 
 \(\bm m_{n,\ell}\) that satisfies \(n/\ell -\frac14 \frac n{L_n^2}\le m_{n,\ell,i}\le n/\ell+\frac14\frac n{L_n^2}\), satisfies 
 \cref{ass:assumptiononthesizeoflabelsets},
 see 
 \cref{eq:testingpower}. A condition like this is quite common in the literature, see for instance \cite[p. 2254]{zhangzhou2016} or the consistency results in \cite[assumption (A1), p. 897]{gengbhattacharyapati2019}.  The difference \(m_{n,\ell_1,\min} - m_{n,\ell_2,\max}\) determines the power to test between \(\Theta_{n,\ell_1}\) and \(\Theta_{n,\ell_2}\), see \cref{thm:posteriorodds} in combination with \cref{prop:postconvThetanl}.

In the planted multi-section model one observes a sequence of
graphs $X^n$ of order $n$ and the statistical question
is to reconstruct the unobserved class assignment vectors $\theta_n$
\emph{consistently}, that is, correctly with probability growing
to one as $n\to\infty$.

\begin{definition}
\label{def:exact}
Let $\theta_{0,n}\in\Theta_n$ for all $n\geq1$ be given. An
estimator sequence $\hat{\theta}_n:\scrX_n\to\Theta_n$ is said
to \emph{recover the class assignment $\theta_{0,n}$ exactly} if,
\[
  P_{\theta_{0,n}}\bigl(\hat{\theta}_n(X^n)=\theta_{0,n}\bigr)\longrightarrow1,
\]
as \(n\to\infty\). That is,  $\hat{\theta}_n$ is the true class assignment (up to a permutation) 
with high probability. 
\end{definition}

In the very sparse regime where the average degree is large enough, or grows arbitrarily slowly to infinity, exact recovery is not possible, however weaker forms of recovery are still possible, for which we define: 
\begin{definition}
\label{def:almostexact}
Let $\theta_{0,n}\in\Theta_n$ for all $n\geq1$ be given. Let  $\hat{\theta}_n:\scrX_n\to\Theta_n$ be an estimator and  \(s_n\) be an integer such that there are representations \( \theta_{0,n}'\) of \( \theta_{0,n}\) and \(  \hat{\theta}_n'\) of \(\hat{\theta}_n\) such that 
\[
  P_{\theta_{0,n}}\rh{\#\set{i:\hat{\theta}_{n,i}'=\theta_{0,n,i}'}\ge s_n}\to1,\quad\text{as}\quad n\to\infty.
\]
that is, if $\hat{\theta}_n$ recovers \(s_n\) labels correctly, with high probability. 
When \(s_n/n\to 1\), we say that \(\hat\theta_n\)  \emph{recovers the class assignment $\theta_{0,n}$ almost exactly}, that is, the estimator recovers the labelling correct, up to an arbitrarily small fraction and up to a permutation of the labels. When \(s_n/n\ge a_n>1/L_n\), for some deterministic sequence \(a_n\), where \(L_n\) is the maximum number of classes in \(\Theta_n\), then we say that \(\hat \theta_n\) recovers the labelling partially. Which means that \(\hat\theta_n\) performs on average better than a random guess. 
\end{definition}

Below we specialize to the Bayesian approach in the planted
multi-section model: with given $n\geq1$, $\theta_n\in\Tht_n$ and
random $X^n$, the likelihood is given by,
\[
  p_{\theta_n}(X^n)=\prod_{1\leq i<j\leq n}
    Q_{i,j;n}(\theta_n)^{X_{ij}}(1-Q_{i,j;n}(\theta_n))^{1-X_{ij}},
\]
where  \(Q_{i,j,n}\) is defined in \cref{eq:pmsundepofrepr}. 

We make the following convenient assumption on the prior (which  always holds after removing parameters with zero prior mass from the parameter space). 
\begin{assumption}\label{eq:prior}
The prior mass function \(\pi_n\) of the prior \(\Pi_n\) on \(\Theta_n\) satisfies \(\pi_n(\theta_n)>0\), for all \(\theta_{ n}\in\Theta_n\). 

\end{assumption}

The posterior distribution of a subset \(S_n\subseteq \Theta_n\) is given by
\[
\Pi_n(S_n\mid X^n)=\frac{\sum_{\theta_n\in S_n}\pi_n(\theta_n) p_{\theta_n}(X^n) }{\sum_{\theta_n\in \Theta_n} \pi_n(\theta_n)p_{\theta_n}(X^n) }.
\]

\newcommand{\upperboundposteriorS}{2
	\rh{
		\frac{
			\pi_n(S_n)
		}{
		\pi_n(\theta_{0,n})}\bigvee |S_n|
	} \rho(p_n,q_n)^{B_n}
}
\begin{proposition}
	\label{prop:postconvset} For fixed \(n\), consider an prior probability mass function \(\pi_n\) on \(\Theta_n\) satisfying \cref{eq:prior}. 
	Suppose that for some \(\theta_{0, n}\in \Theta_n\), we observe a graph
	$X^n$ with $n$ vertices, distributed according to 
	$P_{\theta_{0,n}}.
	$
Let \(S_n\subseteq\Theta_n\weg\set{\theta_{ 0,n}}\) be non-empty.
For \(\theta_{ n}\in S_n\), and representations \(\theta_{ 0,n}'\in \theta_{ 0,n}\) and \(\theta_{ n}'\in \theta_{ n}\) define \[
\begin{split}
D_{1,n}(\theta_n)&=\{(i,j)\in\{1,\ldots,n\}^2:\,i<j,\,
\theta_{0,n,i}'=\theta_{0,n,j}',\,
\theta_{n,i}'\neq\theta_{n,j}'\},\\
D_{2,n}(\theta_n)&=\{(i,j)\in\{1,\ldots,n\}^2:\,i<j,\,
\theta_{0,n,i}'\neq\theta_{0,n,j}',\,
\theta_{n,i}'=\theta_{n,j}'\}.
\end{split}
\]
 Then \(D_{1,n}(\theta_n)\) and \(D_{2,n}(\theta_n)\) are well-defined and disjoint, and if 
 \[
 0<B_n\le \inf_{\theta_n\in S_n}|D_{1,n}(\theta_n)\cup D_{2,n}(\theta_n)|,
 \]then  
\begin{equation}
	\label{eq:ordervstestpwr}
	P_{\theta_{0,n}}
	\Pi_n\bigl(S_n\bigm|X^n\bigr)
	\le \upperboundposteriorS.
	\end{equation}
\end{proposition}
\begin{proof}
Obviously, \(D_{1,n}(\theta_n)\) and \(D_{2,n}(\theta_n)\) are disjoint and do not depend on the chosen representations.
	According to  \cref{lem:posteriorconvergencekleijnlemma}  (with $B_n=\{\theta_{0}\}$),
	for any tests $\phi_{S_n}:\scrX_N\to[0,1]$,
	we have,
	\[
	P_{\theta_{0,n}}\Pi_n(S_n|X^n)
	\le P_{\theta_{0,n}}\phi_{S_n}(X^n)
	+\frac{1}{\pi^{n}_{\theta_{0, n}}} \sum_{\theta_n\in S_n}\pi_n(\theta_n)
	P_{\theta_n}(1-\phi_{S_n}(X^n)).
	\]
	\Cref{lem:testingpower} proves that for any $\theta_n\in S$
	there is a test function $\phi_{\theta_n}$ that distinguishes
	$\theta_{0,n}$ from $\theta_n$ as follows,
	\begin{equation*}
	P_{\theta_{0,n}}\phi_{\theta_n}(X^n)+
	P_{\theta_n}(1-\phi_{\theta_n}(X^n)) \leq
	\rho(p_n,q_n)^{|D_{1,n}(\theta_{ n})\cup D_{2,n}(\theta_n)|}
	\le
	\rho(p_n,q_n)^{B_n},
	\end{equation*}
	where the last inequality follows from the fact that \(\rho(p_n, q_n)\le 1\) and the assumption \(|D_{1,n}(\theta_{ n})\cup D_{2,n}(\theta_n)|\ge  B_n\), for all \(\theta_n\in S_n\). 
	Then using test functions
	$\phi_{S_n}(X^n)=\max\{\phi_{\theta_n}(X^n):\theta_n\in S_n\}$, 
	we have,
	\[
	P_{\theta_{0,n}}\phi_{S_n}(X^n)\le
	\sum_{\theta_n\in S_n}  P_{\theta_{0,n}}\phi_{\theta_n}(X^n),
	\]
	so that,
	\begin{align*}
	P_{\theta_{0,n}}\Pi_n(S_n|X^n) \le & \sum_{\theta_n\in S_n}
	\rh{\frac{\pi_n(\theta_n)}{\pi_{\theta_{0,n}}^{n}}\bigvee 1}\rh{ P_{\theta_{0,n}}\phi_{\theta_n}(X^n) + 
		P_{\theta_n}(1-\phi_{\theta_n}(X^n))}\\
	\le     &  2\rh{\frac{\pi_n(S_n)}{\pi_n(\theta_{0,n})}\bigvee |S_n|} \rho(p_n,q_n)^{B_n},
	\end{align*}
	where we use that for a finite index set \(I\) and a sequence of non-negative numbers \((x_i:i\in I)\),  \(\sum_{i\in  I}(x_i\vee 1)\le 2\rh{|I|\vee \sum_{i\in I}x_i}.\)
\end{proof}

\section{Selection of the number of classes}
\label{sec:SBMselect}
Consider the sequence of experiments in which we observe random
graphs $X^n\in\scrX^n$ generated by the SBM
of \cref{eq:pms}. We first provide a general condition under which the  posterior selects the true model \(\Theta_{n,\ell}\) consistently. Next we apply this to several interesting examples. 
The Hellinger-affinity between two Bernoulli-distributions with
parameters $p$ and $q$, is given by
\[
  \rho(p,q)=p^{1/2}q^{1/2}+(1-p)^{1/2}(1-q)^{1/2}.
\] %
This quantity determines the power of our tests, \cref{lem:testingpower}, similar as in \cite{kleijnwaaij18}.  (See also \cite[equation 1.2 and theorem 1.1]{zhangzhou2016} where minus the log of this quantity (i.e. the R\'enyi divergence of order 1/2) emerges in the minimax rate of convergence.)
 
 \newcommand{\upperboundposteriorThetanl}{2
 	\rh{
 		\frac{\pi_n(\Theta_{n ,\ell})} {\pi_n(\theta_{0,n})}\bigvee |\Theta_{n ,\ell}|} \rho(p_n,q_n)^{\frac12n(m_{n,\ell_0\wedge \ell,\min}-m_{n,\ell_0\vee \ell,\max})
 	}
 }
 
 \newcommand{\inequalityforpostconvThetanl}{
P_{\theta_{0,n}}
\Pi\bigl(\Tht_{n,\ell}\bigm|X^n\bigr)
\le \upperboundposteriorThetanl
}
 
\begin{proposition}\label{prop:postconvThetanl}
\label{prop:sbmselectlm} For fixed \(n\), let the set of models \(M_n\) satisfy  \cref{ass:assumptiononthesizeoflabelsets} and let  \(\Pi_n\) a prior on \(\Theta_n\) satisfying \cref{eq:prior}. 
Suppose that for some $\ell_0\in \sL_n$, we observe a graph
$X^n$ with $n$ vertices, distributed according to 
$P_{\theta_{0,n}}
$
for some class assignment $\theta_{0,n}\in\Theta_{n,\ell_0}$.
For \(\ell\in \sL_n\) unequal to $\ell_0$,
we have,
\begin{equation}
\inequalityforpostconvThetanl.
\end{equation}%
\end{proposition}
\begin{proof}
	This follows from \cref{prop:postconvset,lem:cardinalityofD1ncupD2n}.
\end{proof}
\newcommand{\upperboundposteriorThetanlzerocomplement}{\sum_{{\tiny \begin{array}{c}\ell\in\sL_n,\\ \ell\neq \ell_0\end{array}}}  \upperboundposteriorThetanl}

The number of elements in \(\Theta_{n,\ell}\) is bounded by the number of partitions of \(\set{1,\ldots, n}\) into \(\ell\) sets, which is known as the Stirling number \(V(n,\ell )\) of second kind, which  satisfies 
\begin{equation}\label{eq:upperboundcardinialityThetanl}
V(n,\ell)\le \frac12 \binom n \ell \ell^{n-\ell}\le \frac12 e^\ell n^\ell \ell^{n-2\ell}.
\end{equation}

\subsection{Examples: convergence to the true model} \label{sec:examplesconvergencetothetruemodel}
We consider several examples of posterior contraction in one model \(\Theta_{n,\ell_0}\), whose priors and models are defined below.

\begin{example}[Priors]\label{usepageref:priorexample}
	The first prior is defined hierarchically by taking the uniform prior on \(\sL_n\), and conditionally on \(\ell\), we choose the uniform prior on \(\Theta_{n,\ell}\). So
	\begin{equation}\label{eq:hierarchicalprior}
	\begin{split}
	\ell \sim & \frac 1 {|\sL_n|},\\
	\theta_{0,n}\mid \ell \sim & \frac1{|\Theta_{n ,\ell}|}.
	\end{split}
	\end{equation}
	
	A second option for a prior is the uniform prior on \(\Theta_{n}\), so
	\begin{equation}\label{eq:uniformprior}
	\theta_{ 0,n} \sim \frac1{|\Theta_n|}.
	\end{equation}
	In both cases, applying \cref{prop:postconvThetanl} gives, for \(\ell\in \sL_n, \ell\neq \ell_0\), the upper bound, \begin{equation}\label{eq:upperboundforThetamnlinexamples}
P_{\theta_{0,n}}\Pi_n(\Theta_{n,\ell}|X^n)\le 2\rh{\max_{\ell\in \sL_n} |\Theta_{n ,\ell}|} \rho(p_n,q_n)^{\frac12n(m_{n,\ell_0\wedge \ell,\min}-m_{n,\ell_0\vee \ell,\max})}.
\end{equation} 
	\end{example}

\begin{example}[Models]\label{ex:model}
	We take \(\sL_n=\set{1,\ldots, L_n},\) where \(2\le L_n\le n\) depends on the degree of sparsity and is specified in the examples. We let \(\Theta_{n,\ell}\) be all parameters with \(\ell\) classes and satisfying \begin{equation}\label{eq:modeldescription}
	m_{n,\ell,\min }\ge n/\ell - \frac14 \frac n{L_n^2 } \en  m_{n,\ell,\max }\le n/\ell + \frac14 \frac n{L_n^2 }. 
	\end{equation}
	One verifies easily that  \cref{ass:assumptiononthesizeoflabelsets} is satisfied and that for \(\ell,\ell_0\in \sL_n, \ell\neq \ell_0\),
	\begin{equation}\label{eq:testingpower}
	m_{n,\ell_0\wedge \ell,\min }- m_{n,\ell_0\vee \ell,\max}\ge \frac12 \frac n {L_n^2}.
	\end{equation}
	Using \cref{eq:upperboundcardinialityThetanl}, the number of elements of \(\Theta_{n,\ell}\) is upper bounded by \(\frac12e^{2n\log L_n },\) which is trivial for \(\ell=1\), and when \(\ell\ge 2\) and \(n\ge 4\), we have, %
	\begin{equation}\label{eq:upperboundcardinalityThetanl}
		\frac12 e^\ell n^\ell \ell^{n-2\ell}\le \frac12  n^\ell \ell^{n}\le \frac12  \ell^{2n}\le \frac12e^{2n\log L_n },  
	\end{equation}
	using that for integers \(2\le x\le y\) with either \(x\ge 3\) or \(y\ge 4\), \(y^x\le x^y\) and using that \(\ell\le L_n\le n\).
	It follows that, \begin{equation}\label{eq:posteriorconvergenceinexamples}
	P_{\theta_{0,n}}
	\Pi_n\bigl(\Tht_{n,\ell}\bigm|X^n\bigr)
	\le e^{2n\log L_n}\rho(p_n,q_n)^{\frac14n^2 / (L_n^2) }
	\end{equation}
\end{example}

\begin{example}\label{ex:selectdense}{\it (Dense phase)}
Define \(b_n=-\log \rho(p_n,q_n)>0\). 
Let \(L_n=\bm o(\sqrt{n/\log n})\), for example \(L_n=\floor{n^{\alpha}},\) for some \(\alpha<1/2\). 
When \(\frac{nb_n}{L_n^2\log L_n}\ge 12\), 
 \(
P_{\theta_{0,n}} \Pi(\Theta_{n,\ell}\mid X^n) 
\le   e^{-b_nn^2/(12L_n^2)}.
\)
It follows that 
\(
P_{\theta_{0,n}} \Pi(\Theta_{n,\ell_0}^c\mid X^n)\le L_ne^{-b_nn^2/(12L_n^2)}\le  L_n^{1-n}\to 0,\) as \(n\to\infty. 
\)
\closebox\end{example}

In the sparse regimes another approximation strategy is needed. 

\begin{example} {\it (Chernoff-Hellinger phase)}
\label{ex:selectCH}
Recall that in the Chernoff-Hellinger phase, we assume that the edge probabilities satisfy $np_n=a_n\log(n)$ and
$nq_n=b_n\log(n)$ for sequences $a_n,b_n$ bounded away from \(zero\). Suppose \(L_n\) satisfies \begin{equation}\label{eq:conditiononLninChernoffHellingerPhase}
48L_n^2\log L_n \le (\sqrt{a_n}-\sqrt{b_n})^2\log n. 
\end{equation}
\Cref{eq:conditiononLninChernoffHellingerPhase} is satisfied for sufficiently large \(n\), when \(L_n\le (\log n)^a\) with \(a<1/2\). 

Using \cref{lem:sqrtoneminusxislessthanorequaltooneminusxovertwo}, we find  \(\sqrt{1-p}\sqrt{1-q} \le 1-\frac12 (p+q)+ \frac14pq\), and using the fact that \((1+x/r)^r\le e^x\) (see \cref{lem:oneplusxdivrtothepowerrissmallerthanetothepowerx}), for all positive integers \(r\) and \(x>-r\), we obtain 
\begin{align*}
	&\rho(p_n,q_n)^{\frac12n | m_{n,\ell_0\wedge \ell, \min} - m_{n,\ell_0\vee \ell, \max}|}\\
	\le & \vh{ 1 - \frac12\rh{\sqrt{p_n}-\sqrt{q_n}}^2+ p_nq_n/4}^{\frac12n | m_{n,\ell_0\wedge \ell, \min} - m_{n,\ell_0\vee \ell, \max}|}\\
	= & \vh{ 1 - \frac{\log n}{n}\set{\frac12\rh{\sqrt{a_n}-\sqrt{b_n}}^2- \frac{a_nb_n\log  n}{4n}} }^{\frac12n ( m_{n,\ell_0\wedge \ell, \min} - m_{n,\ell_0\vee \ell, \max})}\\
	\le & \expa{-\frac12( m_{n,\ell_0\wedge \ell, \min} - m_{n,\ell_0\vee \ell, \max})
 	\set{\frac12\rh{\sqrt{a_n}-\sqrt{b_n}}^2- \frac{a_nb_n \log n }{4n}} \log n
}\\
\le & \expa{-\frac1{16}\rh{\sqrt{a_n}-\sqrt{b_n}}^2\frac{n \log n}{L_n^2}},
\end{align*}
for \(n\) sufficiently large. 
It follows from \cref{eq:posteriorconvergenceinexamples} that  
\begin{align*}
&P_{\theta_{0,n}}\Pi(\Theta_{n,\ell}|X^n)
\le  \expa{-n\rh{\frac1{16}\rh{\sqrt{a_n}-\sqrt{b_n}}^2\frac{ \log n}{L_n^2} - 2\log L_n}}\\
\le & \expa{-\rh{\sqrt{a_n}-\sqrt{b_n}}^2\frac{ n\log n}{48L_n^2}
	},
\end{align*}
using \cref{eq:conditiononLninChernoffHellingerPhase} in the last inequality. 
It follows that 
\[
P_{\theta_{0,n}} \Pi(\Theta_{\bm m_{n,\ell_0}}^c\mid X^n)\le  L_n \expa{-\rh{\sqrt{a_n}-\sqrt{b_n}}^2\frac{ n\log n}{48L_n^2}}\le L_n^{1-n}\to 0,  
\] 
 as \(n\to\infty.\)   
\closebox\end{example}

Even in the Kesten-Stigum phase, there is testing power enough to  decide between the classes.

\begin{example} {\it (Kesten-Stigum phase)}\label{ex:kestenstigumphase}
Recall that in the Kesten-Stigum phase, $np_n=c_n$ and
$nq_n=d_n$ for sequences of positive numbers $c_n\) and \(d_n$ bounded away from zero.
We obtain in similar manner as in \cref{ex:selectCH} that, \[(c_n^{1/2}-d_n^{1/2})^2
\ge 48L_n^2\log L_n,\]
 guarantees that, 
\[
P_{\theta_{0,n}} \Pi(\Theta_{n,\ell_0}^c\mid X^n)\le L_n \expa{-\rh{\sqrt{c_n}-\sqrt{d_n}}^2
\frac n{48L_n^2}}\le L_n^{1-n}\to 0,
\]    
as \(n\to\infty\).
\closebox\end{example}

\section{Posterior concentration at the parameter}\label{sec:posteriorconsistencyataparameter}

Posterior convergence at the true parameter \(\theta_{0,n}\in\Theta_{n,\ell_0}\), where \(\ell_0\) is the true number of classes, is derived in the following way. Noting that
\[
P_{\theta_0}\Pi(\set{\theta_{n,0}}^c\mid X^n)= P_{\theta_0} \Pi(\Theta_{n,\ell_0}^c\mid X^n) + P_{\theta_{n,0}}\Pi(\Theta_{n,\ell_0}\backslash\set{\theta_{0,n}}\mid X^n),
\]
posterior consistency is established once both expectations on the right converge to zero as \(n\to\infty\). The first was the content of the previous section, the second we treat here. Obviously, when \(\ell_0=1\), then \(\Theta_{n,\ell_0}=\set{\theta_{n,0}}\) and convergence at the true parameter follows  from \cref{sec:SBMselect}. Hence, we  assume that \(\ell_0>1.\)
In the previous section posterior convergence in a set \(\Theta_{n,\ell_0}\) was established by showing that there was enough testing power between \(\Theta_{n,\ell_0}\) and \(\Theta_{n,\ell}\), for \(\ell\neq \ell_0\). Here we do something similar within \(\Theta_{n,\ell_0}\), by constructing tests between \(\theta_{ 0,n}\) (in the dense and sparse case) and rings \(V_{n,k}\subseteq \Theta_{n ,\ell_0}\) of radius \(k\) around \(\theta_{ 0,n}\) in the distance \(r_n\).

Define the distance \(r_{n}\) on \(\Theta_{n}\) as follows: first define \(r_{n}'\) on \(\Theta_{n}'\) by \[
r_{n}'(\theta',\eta')=\max_{a\neq b}\#\set{i:\theta_i'=a,\eta_i'=b}.
\]
This does depend on the particular choice of the representation. So we define for \(\theta,\eta\in\Theta_n\),
\begin{equation}\label{eq:definitionr}
r_{n}(\theta,\eta)=\min_{\theta'\in\theta,\eta'\in\eta}r_{n}'(\theta',\eta'). 
\end{equation}

The function \(r_{n}\)  takes values in \(\set{0,\ldots,\floor{n/2}}\) and according to \cref{lem:boundonm}, when \(\theta,\eta\in \Theta_{n ,\ell}\),  \(r_n(\theta,\eta)\le  m_{n,\ell,\max}/2\). The function \(r_n\) is symmetric and \(r_n(\theta,\eta)=0\) if and only if \(\theta=\eta\).

For a given \(\ell_0\in\sL_n\) and \(\theta_{0,n}\in\Theta_{n,\ell_0}\) define subsets of \(\Theta_{n,\ell_0} \), by  \[
V_{n,\ell_0,k}=\set{\theta_n\in\Theta_{n,\ell_0}: r_{n}(\theta_n,\theta_{0,n})=k}.
\]
These sets resemble the sets \(V_{n,k}\) in \cite{kleijnwaaij18} and are rings of radius \(k\) around \(\theta_{ 0,n}\) in \(\Theta_{n ,\ell_0}\). When \(k\le m_{n,\ell_0,\min},\) \cref{lem:upperboundVnlk} bounds the number of elements in \(V_{n,\ell_0,k}\) by
\(2^{k\ell_0(\ell_0-1)}\binom{n(\ell_0-1)}{\ell_0(\ell_0-1)k}\), which in turn is bounded by \(
\rh{\frac{2en}{\ell_0k}}^{\ell_0(\ell_0-1)k}.\)

\newcommand{\upperboundposteriorVnlk}{2  \rh{|V_{n,\ell_0,k}| \vee   \frac{\pi_n(V_{n,\ell_0,k})}{\pi_n(\theta_{0, n})}} \rho(p_n,q_n)^{2k(m_{n,\ell_0,\min}-k)^+}}
\begin{proposition}
\label{prop:convergenceatapoint} Consider a prior on \(\Theta_n\) satisfying \cref{eq:prior}. 
Suppose that for some $\ell_0\in \sL_n$, \(\ell_0>1\), we observe a graph
$X^n$ with $n$ vertices, distributed according to 
$P_{\theta_{0,n}}$, where  $\theta_{0,n}\in\Theta_{n,\ell_0}$. For all \(k\in\set{0,\ldots, \floor{m_{n,\ell_0,\max}/2}}\), we have,
\begin{equation}
  P_{\theta_{0,n}}
    \Pi\bigl(V_{n,\ell_0,k}\bigm|X^n\bigr)
    \le \upperboundposteriorVnlk.
\end{equation}
\end{proposition}
\begin{proof}
	This follows from \cref{prop:postconvset} and the lower bound on \(|D_{1,n}(\theta_n)\cup D_{2,n}(\theta_n)|\) in \cref{lem:lowerboundondintermsofr}.
\end{proof}

Define  \[
W_{n,\ell_0,k_n}= \bigcup_{k=k_n}^{\floor{m_{n,\ell\wedge \ell_0,\max}/2}} V_{n,  \ell_0,k},
\]
to be the set of all elements with \(r_n\) distance at  least \(k_n\) from \(\theta_{ 0,n}\). Taking \(k_n=1\), \(W_{n,\ell_0, 1}=\Theta_{n,\ell_0}\backslash\set{\theta_{0,n}}\).

\newcommand{\upperboundposteriorWnlkn}{ \sum_{k=k_n}^{\floor{m_{n, \ell_0,\max}/2}}
	\upperboundposteriorVnlk
}
Applying \cref{prop:convergenceatapoint} to \(W_{n,\ell_0,k_n}\) gives 
\begin{align*}
&P_{\theta_{0,n}}\Pi(W_{n,k_n}|X^n)
\le    \upperboundposteriorWnlkn
\end{align*}

For the next corollary we need the following assumption \begin{assumption}\label{ass:conditiononThetanl}
	For all \(\ell\in \sL_n\), \(m_{n,\ell,\min}\ge m_{n,\ell,\max}/2\). 
\end{assumption}

\begin{corollary}\label{col:convenientcorollaryforposteriorconvergence} Assume \cref{ass:conditiononThetanl}. 
	Suppose that for some $\ell_0\in\sL_n$, \(\ell_0>1\),  and $\theta_{0,n}\in\Theta_{n,\ell_0}$, we observe a graph
	$X^n$ with $n$ vertices, distributed according to 
	$P_{\theta_{0,n}}.$ Futhermore, assume that the prior \(\pi_n\) gives positive mass to every element of \(\Theta_{n,\ell_0}\) and satisfies  \[
	\max_{\theta_n,\eta_n\in\Theta_{n ,\ell_0}}\frac{\pi_n(\theta_n)}{\pi_n(\eta_n)}\le K_{n,\ell_0}.
	\] 
	Define
	\[B_n=2n\rho(p_n,q_n)^{\frac{2m_{n,\ell_0,\min}-m_{n,\ell_0,\max}}{\ell_0(\ell_0-1)}},\] then \begin{equation}\label{eq:upperboundposteriorThetanlzerowithupperboundonpriormassratios}
	P_{\theta_{n,0}}\Pi_n(\Theta_{n,\ell_0}\backslash\set{\theta_{n,0}}\mid X^n)\le 2
	K_{n,\ell_0}B_n^{\ell_0(\ell_0-1)}e^{(\ell_0-1)B_n}.
	\end{equation}
\end{corollary}
\begin{proof}
		Take \(k_n\equiv1\) for  all \(n\) in \cref{prop:convergenceatapoint}. 
	Using \cref{prop:convergenceatapoint}, the upper bound on the number of elements in \(V_{n,\ell,k}\) \cref{eq:upperboundfornumberofelementsinVnkintermsofabinomialexponent}  and  the condition on the prior mass function \(\pi_{n}\) we have the bound \begin{align*}
	&P_{\theta_{n,0}}\Pi(\Theta_{n,\ell_0}\backslash\set{\theta_{n,0}}\mid X^n) \\
	\le &  2K_{n,\ell_0} \sum_{k=1}^{\floor{m_{n,\ell_0,\max}/2}} 2^{k\ell_0(\ell_0-1)}\binom{n(\ell_0-1)}{\ell_0(\ell_0-1)k}\rho(p_n,q_n)^{2k(m_{n,\ell_0,\min}-k)}\\
	\le & 2K_{n,\ell_0} \sum_{k=1}^{\floor{m_{n,\ell_0,\max}/2}} \frac{(2n(\ell_0-1))^{\ell_0(\ell_0-1)k}}{(\ell_0(\ell_0-1)k)!}\rho(p_n,q_n)^{2k(m_{n,\ell_0,\min}-m_{n,\ell_0,\max}/2)}\\
	\le &2K_{n,\ell_0} \sum_{k=\ell_0(\ell_0-1)}^{\infty} \frac{(2n(\ell_0-1))^{k}}{k!}\rho(p_n,q_n)^{k\frac{2m_{n,\ell_0,\min}-m_{n,\ell_0,\max}}{\ell_0(\ell_0-1)}}\\
	\le &2K_{n,\ell_0} \frac{(2n(\ell_0-1))^{\ell_0(\ell_0-1)}}{(\ell_0(\ell_0-1))!}\rho(p_n,q_n)^{2m_{n,\ell_0,\min}-m_{n,\ell_0,\max}} \\
	&	\times 
	\sum_{k=0}^{\infty} \frac{(2n(\ell_0-1))^{k}}{k!}\rho(p_n,q_n)^{k\frac{2m_{n,\ell_0,\min}-m_{n,\ell_0,\max}}{\ell_0(\ell_0-1)}}\\
	= & 2K_{n,\ell_0} \frac1{(\ell_0(\ell_0-1))!} \rh{(2n(\ell_0-1))\rho(p_n,q_n)^{\frac{2m_{n,\ell_0,\min}-m_{n,\ell_0,\max}}{\ell_0(\ell_0-1)}}}^{\ell_0(\ell_0-1)}\\
	&\times\exp\rh{2n(\ell_0-1)\rho(p_n,q_n)^{ \frac{2m_{n,\ell_0,\min}-m_{n,\ell_0,\max}}{\ell_0(\ell_0-1)}}}\\
	= & \frac{2K_{n,\ell_0}}{(\ell_0(\ell_0-1))!}\big((\ell_0-1)B_n\big)^{\ell_0(\ell_0-1)}e^{(\ell_0-1)B_n}.
	\end{align*}
	With the lower bound \(k!\ge \sqrt{2\pi}k^{k+1/2} e^{-k} \) one can easily see that for \(\ell_0\ge 3\), \(\frac{(\ell_0-1)^{\ell_0(\ell_0-1)}}{(\ell_0(\ell_0-1))!}\le 1\), but it also holds for \(\ell_0=2\). Hence 
	\(P_{\theta_{n,0}}\Pi(\Theta_{n,\ell_0}\backslash\set{\theta_{n,0}}\mid X^n)\le2 K_{n,\ell_0}B_n^{\ell_0(\ell_0-1)}e^{(\ell_0-1)B_n}\).
\end{proof}

\subsection{Examples: recovery of the parameters}\label{sec:examplerecoveryofparameters}

In this section we are interested in exact and almost exact recovery for some interesting examples. We consider the same priors and models as in  \cref{usepageref:priorexample,ex:model}. Note that in this case \begin{equation}
m_{n,\ell, \min}-m_{n,\ell,\max}/2 \ge \frac n {8\ell}, 
\end{equation}
so \cref{ass:conditiononThetanl} is satisfied for all \(\ell\in\sL_n\). 

\newcommand{\postconvupperbounddensephase}{P_{\theta_{0,n}}\Pi_n(\theta_0\mid X^n)\ge  1-L_ne^{-b_nn^2/(12L_n^2)}- 2\sqrt ee^{-nb_n/(8L_n)}}
\begin{example}\label{ex:denseconvergenceatapoint}{\it (Dense phase, continuation of \cref{ex:selectdense})}

Recall that \(b_n = -\log \rho(p_n,q_n) \). The quantity 	
 \(B_n\) in \cref{col:convenientcorollaryforposteriorconvergence} is bounded as follows
 \begin{align*}
 B_n= 2n\rho(p_n,q_n)^{\frac{2m_{n,\ell_0,\min}-m_{n,\ell_0,\max}}{\ell_0(\ell_0-1)}}
 \le  2n\rho(p_n,q_n)^{\frac {n}{4\ell_0^2(\ell_0-1)}}
=  e^{-\frac{ n b_n}{ 4\ell_0^2(\ell_0-1) } + \log (2n)}.
 \end{align*}

 So with \cref{col:convenientcorollaryforposteriorconvergence}   and the result of \cref{ex:selectdense}, when \(L_n=\bm o(\sqrt{n/\log n})\) and  \[
nb_n\ge 8L_n^2(L_n-1)\log (2n),
\]
we obtain\[
\postconvupperbounddensephase. 
\]
In particular 
\(
P_{\theta_{0,n}}\Pi_n(\theta_0\mid X^n)\to 1, 
\)
as \(n\to\infty \).
\closebox\end{example}

So in the dense phase we have exponentially fast convergence of the expected posterior mass of \(\theta_{ 0,n}\) to one, under the true distribution. 
In the sparse phases we get only a polynomial rate of convergence.

\begin{example}\label{ex:CHconcentrationatapoint} {\it (Chernoff-Hellinger phase, continuation of \cref{ex:selectCH})}
With a similar calculation as in  \cref{ex:selectCH}, we find,
\begin{align*}
&\rho(p_n,q_n)^{\frac{2 m_{n,\ell_0, \min} - m_{n,\ell_0, \max}}{\ell_0(\ell_0-1)}}\\
= & \vh{ 1 - \frac{\log n}{n}\rh{\frac12\rh{\sqrt{a_n}-\sqrt{b_n}}^2- \frac{a_nb_n\log  n}{4n}} }^{n\frac{2 m_{n,\ell_0, \min} - m_{n,\ell_0, \max}}{\ell_0(\ell_0-1)n}}\\
\le & \expa{-\frac{2 m_{n,\ell_0, \min} - m_{n,\ell_0, \max}}{\ell_0(\ell_0-1)n}
	\rh{\frac12\rh{\sqrt{a_n}-\sqrt{b_n}}^2- \frac{a_nb_n \log n }{4n}} \log n
}.
\end{align*}

It follows that  \[
B_n\le e^{\log (2n)  -\frac{(a_n^{1/2}-b_n^{1/2})^2}{16\ell_0^2(\ell_0-1)}\log n}.\]

So when \((a_n^{1/2}-b_n^{1/2})^2\ge  32L_n^2(L_n-1)\frac{\log(2n)}{\log n}\), 
\begin{align*}
& P_{\theta_{n,0}}\Pi_n(\Theta_{n,\ell_0}\backslash\set{\theta_{n,0}}\mid X^n)
\le  2\sqrt e n^{-(a_n^{1/2}-b_n^{1/2})^2 / (32L_n)}.
\end{align*}

It follows that when \((a_n^{1/2}-b_n^{1/2})^2\ge  48L_n^2(L_n-1)\), then  with the result of  \cref{ex:selectCH}   \begin{align*}
	P_{\theta_{0,n}}\Pi_n(\theta_0\mid X^n)\ge 1- L_n \expa{-n{\frac1{48}\rh{\sqrt{a_n}-\sqrt{b_n}}^2\frac{ \log n}{L_n^2}}}  - 2\sqrt e n^{-(a_n^{1/2}-b_n^{1/2})^2 / (32L_n)}.
\end{align*}
In particular  \(
     P_{\theta_{0,n}}\Pi_n(\theta_{0,n}\mid X^n)\to 1,\) as \(n\to\infty. 
     \)
\closebox\end{example}

In the even sparser Kesten-Stigum phase exact recovery is not possible anymore (\cite{abbesandon2015,abbesandon2018}). Instead we obtain weak recovery, in which all but a small fraction of the labels is recovered. 

\begin{example}\label{ex:kestenstigumphaseconvergenceinaneighbourhood} {\it (Kesten-Stigum phase, continuation of \cref{ex:kestenstigumphase})}
 With a calculation  similar to that of  \cref{ex:selectCH}, we find that, 

\[
\rho(p_n,q_n)^{2k(m_{n,\ell_0,\min}-k)} \le e^{-k(c_n^{1/2}-d_n^{1/2})^2 / (16\ell_0)}.
\]
Using the bound \(2^{\ell_0(\ell_0-1)k}\binom{n(\ell_0-1)}{\ell_0(\ell_0-1)k}\le \rh{\frac{2en}{\ell_0k}}^{\ell_0(\ell_0-1)k}\), we see that the posterior \begin{align*}
P_{\theta_{0,n}}\Pi_n(W_{n,\ell_0,k_n}\mid X^n)\le & 2 \sum_{k=k_n}^{\floor{m_{n,\ell_0,\max}/2}}\rh{\frac{2en}{\ell_0k}e^{-(c_n^{1/2}-d_n^{1/2})^2/(16\ell_0^2(\ell_0-1))}}^{\ell_0(\ell_0-1)k}\\\le &  2\frac{\rh{\frac{2en}{\ell_0k_n}e^{-(c_n^{1/2}-d_n^{1/2})^2 /(16\ell_0^2(\ell_0-1))}}^{\ell_0(\ell_0-1)k_n}}{1-\rh{\frac{2en}{\ell_0k_n}e^{-(c_n^{1/2}-d_n^{1/2})^2 /(16\ell_0^2(\ell_0-1))}}^{\ell_0(\ell_0-1)}}.
\end{align*}
Let \((\delta_n)_{n\ge 1}\) a sequence in \((0,1)\), such that \(\delta_nn\) is an integer and \(\delta_n\) converges to a number  \(\delta\in [0,1)\). 
For \(k_n=\delta_nn\), we see 
\[
P_{\theta_{0,n}}\Pi_n(W_{n,\ell_0,k_n}\mid X^n)\le 2\frac{\rh{2e\delta_n^{-1}\ell_0^{-1}e^{-(c_n^{1/2}-d_n^{1/2})^2/(16\ell_0^2(\ell_0-1))}}^{\delta_n\ell_0(\ell_0-1)n}}{1-\rh{2e\delta_n^{-1}\ell_0^{-1}e^{-(c_n^{1/2}-d_n^{1/2})^2/(16\ell_0^2(\ell_0-1)) }}^{\ell_0(\ell_0-1)}}.
\]

It follows from \cref{lem:inequalityofanepowerdividedbyoneminusanepower} that when \(\delta_nn\ge2\), and \[
\ell_0(\ell_0-1)\vh{-1 -\log 2+\log \delta_n +\log \ell_0+(c_n^{1/2}-d_n^{1/2})^2/(16\ell_0^2(\ell_0-1))}\ge \sqrt{\frac2{\delta_nn}},
\]
\[
P_{\theta_{0,n}}\Pi_n(W_{n,\ell_0,k_n}\mid X^n)\le 2 e^{-\delta_n\ell_0(\ell_0-1)\vh{-1 -\log 2+\log \delta_n +\log \ell_0+(c_n^{1/2}-d_n^{1/2})^2/(16\ell_0^2(\ell_0-1))}n/4}.
\]
We define the Hamming metric \(m_n\), closely related to \(r_n\), on the parameter space, again stepwise via a metric \(m_n'\) on \(\Theta_n'\). Define  \(m_n'\) on \(\Theta_n'\), by \[
m_n'(\theta',\eta')=\sum_{i=1}^n \mathbb{I}_{\theta_i'\neq \eta_i'},
\]
and next we define \begin{equation}\label{eq:definitionmn}
m_n(\theta,\eta)=\min_{\theta'\in\theta,\eta'\in\eta} m_n'(\theta',\eta')
\end{equation}
on \(\Theta_n\). The quantity \(m_n(\hat \theta_n,\theta_{0,n})\) counts the number of misspecified labels of an estimator \(\hat \theta_{ n}\). 
Let \[
\bar B_{n,k_n}(\theta_n)=\set{\eta_n \in\Theta_{n,\ell_0}:  r_{n}(\theta_n,\eta_n)\le k_n 
}\]
and 
\begin{equation}\label{eq:definitionballbnkn}
B_{n,k_n}(\theta_n)=\set{\eta_n \in\Theta_n:  m_n(\theta_n,\eta_n)\le k_n}.
\end{equation}

It follows from  \cref{eq:equivalencemandr} that \begin{equation}\label{eq:BsubsetofbarB}
\bar B_{n,k_n/(\ell_0(\ell_0-1))}(\theta_n)\subseteq B_{n,k_n}(\theta_n) \quad\text{and clearly}\quad  W_{n,\ell_0,k_n}=\Theta_{n,\ell_0}\weg \bar B_{n,k_n-1}(\theta_{0,n}).
\end{equation} %

Let now \(\delta_n\) be a sequence of positive numbers such that \(\delta_nn/[\ell_0(\ell_0-1)]\) are integers, \(\delta_n/[\ell_0(\ell_0-1)]\in(0,1)\) and decreases to a number \(\delta\)  in \([0,1)\), then 
\begin{align*}
&P_{\theta_{0,n}}\Pi_n(\Theta_{m_{n,\ell_0}}\weg B_{n,\delta n}\mid X^n)\\
\le & P_{\theta_{0,n}}\Pi_n(\Theta_{m_{n,\ell_0}}\weg \bar B_{n, \delta_nn/[\ell_0(\ell_0-1)]-1}\mid X^n)\\
= & P_{\theta_{0,n}}\Pi_n(W_{n,\ell_0,\delta_nn/[\ell_0(\ell_0-1)]}\mid X^n)\\
 \le &  2 e^{-\delta_n\vh{-1 -\log 2 +\log \delta_n - \log (\ell_0-1)+(c_n^{1/2}-d_n^{1/2})^2/(16\ell_0^2(\ell_0-1))}n/4}.
\end{align*}

When \(\delta_n\to 0\), \(-\log \delta_n\to\infty\), so with the result of \cref{ex:kestenstigumphase} almost exact convergence is achieved once \((c_n^{1/2}-d_n^{1/2})^2\to \infty\) (however slowly).
 To be exact, for \((c_n^{1/2}-d_n^{1/2})^2\ge 48 L^2\log L\),  \(\frac{\delta_nn}{L(L-1)}\ge2\), and \begin{equation}\label{eq:sufficientconditionforposteriorconditioninKSphase}
-1 -\log 2+\log \delta_n -\log( L-1)+(c_n^{1/2}-d_n^{1/2})^2/(16L^2(L-1))\ge \sqrt{\frac{2L(L-1)}{\delta_nn}},
\end{equation}

\begin{align*}
	P_{\theta_{0, n}}\Pi_n(B_{n,\delta n})\mid X^n) \le  1- L_n \expa{-n{\frac1{48L^2}(\sqrt{c_n}-\sqrt{d_n})^2
	}}\\
- 2 e^{-\delta_n\rh{-1 -\log 2+\log \delta_n -\log( L-1)+(c_n^{1/2}-d_n^{1/2})^2/(16L^2(L-1))}n/4}.
\end{align*}

Note that condition \cref{eq:sufficientconditionforposteriorconditioninKSphase} is in particular satisfied when \[
\delta_n:=2(L-1)e^{2-(c_n^{1/2}-d_n^{1/2})^2/(16L^2(L-1))}
\]
and  \((c_n^{1/2}-d_n^{1/2})^2\) converges to infinity. Then \(\delta_n\to0\) and \[
P_{\theta_{0, n}}\Pi_n(B_{n,\delta n}\mid X^n) \ge  1-  L_n \expa{-n{\frac1{48L^2}\rh{\sqrt{c_n}-\sqrt{d_n}}^2
}}- 2 e^{-\delta_n n/4}. 
\]

For fixed \(\delta_0<1/L\), setting \((c_n^{1/2}-d_n^{1/2})^2\) so that \(\delta_0=2(L-1)e^{2-(c_n^{1/2}-d_n^{1/2})^2/(16L^2(L-1))}\),  and \(\delta_n=\floor{2(L-1)ne^{2-(c_n^{1/2}-d_n^{1/2})^2/(16L^2(L-1))}}/n\) gives concentration in a ball of radius \(\delta_0n\). 
\closebox\end{example}

\section{Coverage of credible sets with examples}\label{sec:examplecrediblesets}

Conditionally on an observation \(X^n\), a credible set of credible level \(1-\alpha_n\) is a measurable subset \(D_n(X^n)\) of the parameter set with posterior mass at least \(1-\alpha_n\):
\[
\Pi_n(D_n(X^n)\mid X^n)\ge 1-\alpha_n.
\]
In our (discrete, finite) setting any set-valued map \(x^n\mapsto B_n(x^n)\subseteq\Theta_n\), the corresponding map \(x^n\mapsto \Pi_n(B_n(x^n)\mid x^n)\) is measurable and positive, and hence the integral  \(P_{\theta_0}\Pi_n(B_n(X^n)\mid X^n)\) is well-defined, see \cref{app:defs} for details. From this perspective, a credible set (of confidence level $1-\alpha_n$) is a set-valued map \(x^n\mapsto D_n(x^n) \) satisfying \(\Pi_n(D_n(x^n)\mid x^n)\ge 1-\alpha_n,\) for every \(x^n\in\scrX_n\). 
In nonparametric setting, credible sets can have bad coverage: \cite{freedman1999} provides us with examples. However in this section we 
show that in the case of exact recovery credible sets cover \(\theta_{ 0,n}\) with high probability.
 In case of almost exact recovery we make the credible sets larger in order to guarantee asymptotic coverage, using ideas of \cite{kleijn2016} and \cite{kleijnwaaij18}.

\begin{lemma}\label{lem1}
	Let \(n\ge1\). Let \(x^n\to B_n(x^n)\subset\Theta_n\) be a set valued map, such that  \(P_{\theta_0}\Pi_n(B_n(X^n)\mid X^n)\ge 1-a_n\), with \(0<a_n<1\).
	 Then, for every \(0<r_n<1\), \[
	P_{\theta_0}\left[\Pi_n(B_n(X^n)\mid X^n)\ge 1-r_n\right]\ge  1-\frac1{r_n}a_n.\]
\end{lemma}
\begin{proof}	 
	Let \(E_n=\set{\omega:\Pi_n(B_n\mid X^n(\omega))\ge 1-r_n}\) 
	be the event that the posterior mass of \(B_n(X^n)\) is at least \(1-r_n\). Let \(\delta>0\). Suppose that  \(P_{\theta_0}(E_{n})\le 1-\frac1{r_n}a_n-\delta\). Then \begin{align*}
	P_{\theta_0}\Pi_n(B_n(X^n)\mid X^{n})\le P_{\theta_0}( E_{n})+(1-r_n)P_{\theta_0}(E_{n}^c)
	= P_{\theta_0}( E_{n}) + (1-r_n) (1-P_{\theta_0}( E_{n}))\\
	= r_nP_{\theta_0}( E_{n}) + 1-r_n \le r_n\rh{1-\frac1{r_n}a_n-\delta} + 1-r_n=1-a_n-\delta r_n<1-a_n,
	\end{align*}
	which contradicts with our assumption that  \(P_{\theta_0}\Pi_n(B_n(X^n)\mid X^n)\ge 1-a_n\). Hence \(P_{\theta_0}(E_{n})> 1-\frac1{r_n}a_n-\delta\). As this holds for every \(\delta>0\), it follows that \(P_{\theta_0}(E_{n})\ge 1-\frac1{r_n}a_n\).
\end{proof}

\begin{lemma}\label{lem:crediblesettoconfidencesets}
	Suppose \(P_{\theta_0}\Pi_n(\set{\theta_{0,n}}\mid X^n)\ge 1-x_n\), where \(0<x_n<1\). Let \(\alpha_n\in(0,1)\) and \(D_n(X^n)\) a \( 1-\alpha_n\) credible set, i.e. \(\Pi_n(D_n(X^n)\mid X^n)\ge 1-\alpha_n\). Then \[P_{\theta_{0, n}}(\theta_{0, n}\in D_n(X^n))\ge  1-\frac1{1-\alpha_n}x_n.\] 
\end{lemma}
\begin{proof}
	Let \(E_n=\set{\Pi(\set{\theta_0}\mid X^n)\ge r}\) be the event that \(\set{\theta_0}\) has posterior mass at least \(r\), \(r>\alpha_n\). 
	It follows from \cref{lem1} that \(P_{\theta_{0, n}}(E_n)\ge  1-\frac1{1-r}x_n\).
	As \(D_n(X^n)\) has at least \(1-\alpha_n\) posterior mass, \(D_n(X^n)\) and \(\set{\theta_0}\) cannot be disjoint on the event \(E_n\), as \(1-\alpha_n+r>1\). In other words, \(\theta_0\in D_n(X^n)\) on \(E_n\). So \(P_{\theta_0}(\theta_0\in D_n(X^n))\ge P_{\theta_0}(E_n)\ge  1-\frac1{1-r}x_n\). As this holds for any \(r>\alpha_n \) we have \(P_{\theta_0}(\theta_0\in D_n(X^n))\ge 1-\frac1{1-\alpha_n}x_n\).
\end{proof}

\begin{example}\label{ex:densecrediblesets}{\it (Dense phase, continuation of \cref{ex:selectdense,ex:denseconvergenceatapoint})}
	Let \(D_n(X^n) \) be a \(1-\alpha_n\) credible set, \(0<\alpha_n<1\). With  \cref{ex:denseconvergenceatapoint} and \cref{lem:crediblesettoconfidencesets} we get \begin{align*}
	P_{\theta_{0}}(\theta_{0,n}\in D_n(X^n))\ge 1- \frac{1}{1-\alpha_n}   \rh{
L_ne^{-b_nn^2/(12L_n^2)} +  2\sqrt ee^{-nb_n/(8L_n)}	
}.
	\end{align*}
\closebox\end{example}

\begin{example}\label{ex:CHcredibleset} {\it (Chernoff-Hellinger phase, continuation of \cref{ex:selectCH,ex:CHconcentrationatapoint})}
	Let \(D_n(X^n) \) be a \(1-\alpha_n\) credible set, \(0<\alpha_n<1\). With  \cref{ex:CHconcentrationatapoint} and \cref{lem:crediblesettoconfidencesets} we get \begin{align*}
	P_{\theta_{0}}(\theta_{0,n}\in D_n(X^n))\ge 1- \frac{1}{1-\alpha_n}   \bigg[L_n \expa{-n{\frac1{48}(\sqrt{a_n} + \sqrt{b_n})^2\frac{ \log n}{L_n^2}}} & \\+ 2\sqrt e n^{-(a_n^{1/2}-b_n^{1/2})^2 / (32L_n)}\bigg]&.
	\end{align*}
\closebox\end{example}

\subsection{Enlarged credible sets}

In the case of almost exact convergence, credible sets need to be enlarged, in order to make them asymptotic confidence sets. 

Let \(n\ge 1\), let \(D_n(X^n)\) be a credible set. For a  nonnegative integer \(k_n\), we define the \(k_n\)-enlargement of \(D_n(X^n)\) to be the set  \[
C_n(X^n)=\set{\theta_n\in\Theta_n:\exists \eta_n\in D_n(X^n), m_n(\theta_n,\eta_n)\le k_n}.
\]

Recall the definition of \(B_{n,k_n}(\theta_{0,n})\) in \cref{eq:definitionballbnkn}, 
\[
B_{n,k_n}(\theta_n)=\set{\eta_n \in\Theta_n:  m_n(\theta_n,\eta_n)\le k_n}.
\]

We have the following result

\begin{lemma}\label{lem:coverageofenlargedcrediblesets}
Suppose \(P_{\theta_0}\Pi_n(B_{n,k_n}(\theta_{0,n})\mid X^n)\ge1-x_n\), \(0<x_n<1\). Let \(\alpha_n\in(0,1)\) and \(D_n(X^n)\) a \(1-\alpha_n\)-credible set, with \(k_n\)-enlargement \(C_n\), then \[P_{\theta_{0,n}}(\theta_{0,n}\in C_n(X^n))\ge 1-\frac1{1-\alpha_n}x_n.\]. 
\end{lemma}
\begin{proof}
	Let \(E_n=\set{\Pi(B_{n,k_n}(\theta_{0,n})\mid X^n)\ge r}\) be the event that \(B_{n,k_n}(\theta_{0,n})\) has posterior mass at least \(r\), \(r>\alpha_n\). 
	It follows from \cref{lem1} that \(P_{\theta_{0, n}}(E_n)\ge  1-\frac1{1-r}x_n\). 
	As \(D_n(X^n)\) has at least \(1-\alpha_n\) posterior mass, \(D_n(X^n)\) and \(B_{n,k_n}(\theta_{0,n})\) cannot be disjoint on the event \(E_n\), as \(1-\alpha_n+r>1\). Hence \(\theta_0\in C_n(X^n)\) on \(E_n\). So \(P_{\theta_0}(\theta_0\in C_n(X^n))\ge P_{\theta_0}(E_n)\ge  1-\frac1{1-r}x_n\). As this holds for any \(r>\alpha_n \) we have \(P_{\theta_0}(\theta_0\in C_n(X^n))\ge 1-\frac1{1-\alpha_n}x_n\).
\end{proof}

\begin{example}\label{ex:kestenstigumphasecredibleset} {\it (Kesten-Stigum phase, continuation of \cref{ex:kestenstigumphase,ex:kestenstigumphaseconvergenceinaneighbourhood}.)}
 Let  \(D_n(X^n)\) a \(1-\alpha_n\)-credible set and \(C_n(X^n)\) the \(a_n n\) enlargement of \(D_n(X^n)\), with
 \(a_n:=2(L-1)e^{2-\frac1{4L}(c_n^{1/2}-d_n^{1/2})^2}.\)
 With \cref{ex:kestenstigumphaseconvergenceinaneighbourhood} and \cref{lem:coverageofenlargedcrediblesets} we get \[P_{\theta_{0, n}}(\theta_{0,n}\in C_n)\ge 1- \frac1{1-\alpha_n}\rh{   L_n \expa{-n{\frac1{48L^2}\rh{\sqrt{c_n}-\sqrt{d_n}}^2
 	}} +  2 e^{-a_n n/4}} .\]
\closebox\end{example}

\section{Consistent hypothesis testing with posterior odds}\label{sec:hypothesistesting}

Besides parameter estimation, an interesting question is testing between two alternatives, whether the true parameter \(\theta_{ 0,n}\) is in the set \(A_n\) or in the set \(B_n\), where \(A_n,B_n\subseteq \Theta_n\) are disjoint non-random sets. In particular we consider symmetric testing between two alternatives 
\[
H_0:\theta_{ 0,n}\in A_n\quad \text{versus}\quad H_1:\theta_{ 0,n}\in B_n.
\]Taking, for example, \(A_n= \Theta_{n ,\ell_0}\) and \(B_n=\Theta_n\weg\Theta_{n ,\ell_0}\) allows us to test whether the true parameter has \(\ell_0\) classes or not. The case \(\ell_0=1\)  is testing between the Erd\H os-R\'enyi model and the stochastic block model. 
We establish frequentist results for posterior odds testing between \(A_n\) and \(B_n\).%

We use posterior odds to test between the models, which is defined by \[
F_n=\frac{\Pi_n(B_n\mid X^n)}{\Pi_n(A_n\mid X^n)}.
\]
 Obviously, \(F_n<1\) counts as evidence in favour of \(H_0\) and \(F_n>1\) as evidence in favour of \(H_1\). In the following theorem we give sufficient conditions for this Bayesian test to be  valid in a frequentist sense.

\begin{theorem}\label{thm:posteriorodds}
 Let \(\theta_{0,n}\in \Theta_n\). 	When   
	\(P_{\theta_{0,n}}\Pi_n(A_n\mid X^n)\ge  1-a_n\), with \(0<  a_n<1\), 
	then
	\[
	P_{\theta_{0, n}}(F_n> r_n)\le 
	2a_n\rh{1+    \frac{1 }{r_n}}.
	\]
	If, in addition,  \(P_{\theta_{0,n}}\Pi_n(B_n\mid X^n)\le  b_n\), 
	then
	\[
	P_{\theta_{0, n}}(F_n> r_n)\le 2a_n+\frac{2b_n}{r_n}.\]
\end{theorem}
\begin{proof}
	From the posterior convergence condition on \(A_n\) it follows that \(P_{\theta_{0,n}}\Pi_n(B_n\mid X^n)\le  a_n\). Hence the first result follows from the second, so we assume \(P_{\theta_{0,n}}\Pi_n(B_n\mid X^n)\le  b_n\) in what follows. 
	Let \(E_n=\set{\Pi_n(A_n\mid X_n)\ge 1/2}\) be the event that the posterior gives at least mass \(1/2\)  to \(A_n\). It follows from \cref{lem1} that \(P_{\theta_{0, n}}(E_n)\ge 1-2a_n.\)
	So	\begin{align*}
	P_{\theta_{0, n}}(F_n>r_n) 
	\le & P_{\theta_{0, n}} \rh{ \Pi_n(B_n\mid X_n)\ge r_n/2 } + 2a_n.
\end{align*}  
The probability on the right is by the Markov inequality bounded by \[\frac 2{r_n}P_{\theta_{0,n}}\Pi_n(B_n\mid X_n)\le \frac{2b_n}{r_n}.\] We thus arrive at the result
	\[
	P_{\theta_{0, n}}(F_n> r_n)\le 
	2a_n+\frac{2b_n}{r_n}.
	\] 
\end{proof}
Suppose one rejects the null-hypothesis when \(F_n>r_n\), for some \(r_n>0\). The first order error is when \(H_0\) is true, so \(\theta_{0,n}\) is in fact in \(A_n\), but \(H_0\) is rejected (so \(F_n>r_n\)). The probability of this error is bounded by the theorem above. The error of second kind is when in fact \(H_1\) is true, but \(H_0\) is not rejected. This probability is given by \(P_{\theta_{0,n}}(F_n<r_n),\) \(\theta_{0,n}\in B_n\). As \(P_{\theta_{0,n}}(F_n<r_n)=P_{\theta_{0,n}}(F_n^{-1}>1/r_n),\) and reversing the roles of \(A_n \)
and \(B_n\) in \cref{thm:posteriorodds}, the probability of this event is also covered by the theorem, using posterior convergence results for \(\theta_{ 0,n}\in B_n\). The power of the test is defined as the probability of rejecting the null hypothesis when \(H_1\) is true.
As 
\(P_{\theta_{0, n}}(F_n>r_n)=P_{\theta_{0, n}}(F_n^{-1}<1/r_n)=1-P_{\theta_{0, n}}(F_n^{-1}\ge 1/r_n)\), this probability can be lower bounded with the theorem above. 

\subsection{Examples: consistent Bayesian testing}\label{sec:examplesbayesiantesting}

In this section we determine conditions for Bayesian testing for different sparsity regimes.   We consider the same priors and models as in example  \cref{sec:examplesconvergencetothetruemodel}. 

We consider testing
\[
H_0: \theta_{ 0,n}\in\Theta_{n,\ell_0}\quad\text{versus}\quad 
H_1: \theta_{ 0,n}\in\Theta_n\weg\Theta_{n,\ell_0}.
\]

So we consider testing whether the true parameter has \(\ell_0\) classes or not. The case \(\ell_0=1\) is testing between the Erd\H os-R\'enyi graph and the stochastic block model.
\begin{example}\label{ex:selectdensetesting}{\it (Dense phase, continuation of \cref{ex:selectdense})}
It follows from \cref{ex:selectdense} and \cref{thm:posteriorodds} that 
	\[
	P_{\theta_{0, n}}(F_n>r_n) \le 2L_n\rh{1+\frac1{r_n}} e^{-b_nn^2/(12L_n^2)}. 
	\]  
	Taking \(r_n\equiv 1 \), the error of first kind is bounded by \(4L_ne^{-b_nn^2/(12L_n^2)}\).
	Reversing the roles of \(A_n\) and \(B_n\) in \cref{thm:posteriorodds} gives the error of second kind is bounded by \( 4e^{-b_nn^2/(12L_n^2)}.\)
	 and the power of the test is lower bounded by \(1-4e^{-b_nn^2/(12L_n^2)}\).
\closebox\end{example}

\begin{example} {\it (Chernoff-Hellinger phase, continuation of \cref{ex:selectCH})}
	\label{ex:selectCHtoetsen}
	With \cref{ex:selectCH} and \cref{thm:posteriorodds} we get for \(\theta_{ 0,n}\in\Theta_{n ,\ell_0}\)
\[
	P_{\theta_{0, n}}(F_n>r_n) \le 2L_n\rh{1+\frac1{r_n}}\expa{-\frac n {48}\rh{\sqrt{a_n}-\sqrt{b_n}}^2\frac{ \log n}{L_n^2}
	}
	\]
	and for \(\theta_{ 0,n}\in \Theta_n\weg \Theta_{n ,\ell_0}\),  
	\[
	P_{\theta_{0, n}}(F_n^{-1}>r_n^{-1}) \le 2\rh{1+r_n}\expa{-\frac n {48}\rh{\sqrt{a_n}-\sqrt{b_n}}^2\frac{ \log n}{L_n^2}
	},
	\]
	which  is the bound for the second order error.
	The power for the test is lower bounded by \[
	1- 2\rh{1+r_n}\expa{-{\frac n{48}\rh{\sqrt{a_n}-\sqrt{b_n}}^2\frac{ \log n}{L_n^2}
	}}.
	\]
\closebox\end{example}

\begin{example} {\it (Kesten-Stigum phase, continuation \cref{ex:kestenstigumphase})}\label{ex:kestenstigumphasetoetsen}
	 It follow from \cref{ex:kestenstigumphase} and \cref{thm:posteriorodds} that 
	\[
	P_{\theta_{0,n}} (F_n>r_n) \le 2L\rh{1+\frac1{r_n}}   \expa{-{\frac n{48L^2}\rh{\sqrt{c_n}-\sqrt{d_n}}^2}}.
	\]
	And similar results as in the example above hold for the second order error and the power of the test.  
\closebox\end{example}

\appendix

\section{Definitions and conventions}
\label{app:defs}

\paragraph{Notation}
	When \(S\) is a set, \(|S|\) denotes the cardinality of \(S\). %

We assume for every $n\geq1$, a random graph
$\samplen$ taking values in the (finite) space
$\scrX_n$ of all undirected simple graphs (i.e. no self-loops or multiple edges) with $n$ vertices.
Let $\scrB_n$ be the powerset of $\scrX_n$, be the \(\sigma\)-algebra corresponding to $\scrX_n$ and let \(\scrP_n\) be the set of all probability distributions
$P_n:\scrB_n\to[0,1].$ A model is a subset $\sP_n$ of $\scrP_n$, which  is  parametrized by an bijective mapping 
$\Theta_n\rightarrow\sP_n:\theta\mapsto P_{\theta_n}.$ We equip  \(\Theta_n\) with a \(\sigma\)-algebra $\scrG_n$ and a probability measure \(\Pi_n:\scrG_n\to[0,1]\) (the so-called prior). As we only consider finite parameters sets \(\Theta_n\), we set \(\scrG_n\) to be the powerset of $\Theta_n$. As frequentists, we assume that there exists
a `true, underlying distribution for the data';  that
means that for every $n\geq1$,  the 
$n$-th graph $\samplen$ is drawn from $P_{\tht_{0,n}}$, for some $\tht_{0,n}\in\Tht_n$. We call $\tht_{0,n}$ the `true parameter'. 

\begin{definition}
Given $n\geq1$ and a prior probability measure $\Pi_n$ on
$\Tht_n$, define the \emph{$n$-th prior predictive distribution}
as:
\begin{equation}
  \label{eq:priorpred}
  P_n^{\Pi_n}(A) = \int_\Theta P_{\theta_n}(A)\,d\Pi_n(\theta),
\end{equation}
for all $A\in\scrB_n$. %
\end{definition}
The prior predictive distribution $P_n^{\Pi_n}$ is the marginal
distribution for $\samplen$ in the Bayesian perspective that
considers parameter and sample jointly
$(\theta,\samplen)\in\Theta\times\scrX_n$
as the random quantity of interest. 
\begin{definition}
\label{def:posterior}
Given $n\geq1$,
a \emph{(version of) the posterior} is any map
$\Pi_n(\,\cdot\,|\samplen=\,\cdot\,):\scrG_n\times\scrX_n\rightarrow[0,1]$
such that,
\begin{enumerate}
\item for every $B_n\in\scrG_n$, the map
  $\scrX_n\rightarrow[0,1]:
    \realizationn\mapsto\Pi(B_n|\samplen=\realizationn)$ is
  $\scrB_n$-measurable,
\item for all $A_n\in\scrB_n$ and $V_n\in\scrG_n$,
  \begin{equation}
    \label{eq:disintegration}
    \int_{A_n}\Pi_n(V_n|\samplen)\,dP_n^{\Pi_n} = 
    \int_{V_n} P_{\theta_n}(A_n)\,d\Pi_n(\theta_n).
  \end{equation}
\end{enumerate}
\end{definition}
Bayes's Rule is expressed through equality (\ref{eq:disintegration})
and is sometimes referred to as a `disintegration' (of the joint
distribution of $(\theta_n,X^n)$). 
Because the models $\scrP_n$ are dominated (denote the density
of $P_{\theta_n}$ by $p_{\theta_n}$), the fraction of integrated likelihoods,
\begin{equation}
  \label{eq:posteriorfraction}
  \Pi(V_n|\samplen)= 
  {\displaystyle
  	{\int_{V_n} p_{\theta_n}(\samplen)\,d\Pi_n(\theta_n)}} \biggm/
  {\displaystyle{\int_{\Theta_n} p_{\theta_n}(\samplen)\,d\Pi_n(\theta_n)}},
\end{equation}
for all $V_n\in\scrG_n$, $n\geq1$. %

For completeness sake, we include \cite[lemma 2.2]{kleijn2016}, which plays an essential role in our theorems on  posterior consistency.

\begin{lemma}\label{lem:posteriorconvergencekleijnlemma}
	For any \(B_n, V_n\in \scrG_n\) with \(\Pi_n(B_n)>0\) and any measurable map \(\phi_n:\scrX_n \to[0,1] \), \begin{align*}
	\int P_{\theta_n}\Pi_n(V_n\mid X^n)d\Pi_n(\theta_n\mid B_n)\le &  \int P_{\theta_n}[\phi_n(X^n)]d\Pi_n(\theta_n\mid B_n) \\
	& + \frac1{\Pi_n(B_n)} \int _{V_n}P_{\theta_n}[1-\phi_n(X^n)]d\Pi_n(\theta_n). 
	\end{align*}
\end{lemma}

\section{Existence of suitable tests}
\label{app:PBMtests}

Given $n\geq1$, and two class assignment vectors
$\theta_{0,n},\theta_n\in\Theta_n$,
we are interested in determining testing power, for which we need the likelihood ratio
$dP_{\theta_n}/dP_{\theta_{0,n}}$.

Fix $n\geq1$, and let $X^n$ denote the random graph associated with
$\theta_{0,n}\in\Theta_n$, and let \(\ell_0\) be the number of different labels of \(\theta_{ 0,n}\), so \(\theta_{0,n}\in \Theta_{n,\ell_0}\). Let $\theta_n$ denote another element of $\Theta_n$ and  suppose \(\theta_n\in \Theta_{n ,\ell}\), for some \(\ell \in \sL_n\) (which might or might not be equal to \(\ell_0\)). 
Compare $p_{\theta_{0,n}}(X^n)$ with $p_{\theta_n}(X^n)$ in the likelihood
ratio. Recall that the likelihood of $\theta_n$ is given by,
\[
  p_{\theta_n}(X^n)=\prod_{i<j} Q_{i,j;n}(\theta_n')^{X_{ij}}
    (1-Q_{i,j;n}(\theta_n'))^{1-X_{ij}},
\]
where \(\theta_{ n}'\) is a representation of \(\theta_{ n}\) and 
\[
Q_{i,j;n}(\theta_n') = \begin{cases}
	\,\,p_n,&\quad\text{if $\theta_{i}'=\theta_{j}'$,}\\
	\,\,q_n,&\quad\text{if $\theta_{i}'\neq\theta_{j}'$.}
\end{cases}
\]
Let \(\theta_{ 0,n}'\) be a representation of \(\theta_{ 0,n}\) and \(\theta_n'\) a representation of \(\theta_n\). 
Define two sets of edges, one consisting of all edges that connect
within a class under $\theta_{0,n}'$ and between classes under
$\theta_n'$, and another consisting of all edges that connect
between classes under $\theta_{0,n}'$ and within a class under
$\theta_n'$:
\[
  \begin{split}
    D_{1,n}&=\{(i,j)\in\{1,\ldots,n\}^2:\,i<j,\,
      \theta_{0,n,i}'=\theta_{0,n,j}',\,
      \theta_{n,i}'\neq\theta_{n,j}'\},\\
    D_{2,n}&=\{(i,j)\in\{1,\ldots,n\}^2:\,i<j,\,
      \theta_{0,n,i}'\neq\theta_{0,n,j}',\,
      \theta_{n,i}'=\theta_{n,j}'\}.
  \end{split}
\]
Note that \(D_{1,n}\) and \(D_{2,n}\) do not depend on the chosen representations of \(\theta_{ 0,n}\) and \(\theta_n\). 
 Also define,
\[
  (S_n,T_n):=\Bigl(\sum\{X_{ij}:(i,j)\in D_{1,n}\},
    \sum \{X_{ij}:(i,j)\in D_{2,n}\}\Bigr),
\]
and note that, under $P_{\theta_{0,n}}$ and $P_{\theta_n}$,
\begin{equation}
  \label{eq:SnTn}
  (S_n,T_n)\sim\begin{cases}
  \text{Bin}(|D_{1,n}|,p_n)\times\text{Bin}(|D_{2,n}|,q_n),
    \quad\text{if $X^n\sim P_{\theta_{0,n}}$},\\
  \text{Bin}(|D_{1,n}|,q_n)\times\text{Bin}(|D_{2,n}|,p_n),
    \quad\text{if $X^n\sim P_{\theta_n}$}.
  \end{cases}
\end{equation}
Since $S_n$ and $T_n$ are independent, the likelihood ratio is fixed
as a product two exponentiated binomial random variables: 
\begin{equation}
  \label{eq:pbmlikratio}
  \frac{p_{\theta_n}}{p_{\theta_{0,n}}}(X^n)
  = \biggl(\frac{1-p_n}{p_n}\,\frac{q_n}{1-q_n}\biggr)^{S_n-T_n}
    \biggl(\frac{1-q_n}{1-p_n}\biggr)^{|D_{1,n}|-|D_{2,n}|}
\end{equation} %
This gives rise to the following lemma:
\begin{lemma}
\label{lem:testingpower}
Let $n\geq1$, $\theta_{0,n},\theta_n\in\Theta_n$ be given. Then there
exists a test function $\phi_n:\scrX_n\to[0,1]$ such that,
\[
  P_{\theta_{0,n}}\phi_n(X^n) + P_{\theta_n}(1-\phi_n(X^n))\leq \rho(p_n,q_n)^{|D_{1,n}|+|D_{2,n}|}.
\]
\end{lemma}
\begin{proof}
The likelihood ratio test $\phi_n(X^n)$ has testing power bounded by
the Hellinger transform,
\[
  P_{\theta_{0,n}}\phi_n(X^n) + P_{\theta_n}(1-\phi_n(X^n))\leq
    P_{\theta_{0,n}}\Bigl(\frac{p_{\theta_n}}{p_{\theta_{0,n}}}(X^n)\Bigr)^{1/2},
\]
(see, \eg\ \cite{LeCam86} and \cite[lemma 2.7]{kleijn2016}).
Then
\[
  \begin{split}
  P_{\theta_{0,n}}\biggl(\frac{p_{\theta_n}}{p_{\theta_{0,n}}}(X^n)\biggr)^{1/2}
    &= P_{\theta_{0,n}}
    \biggl(\frac{p_n}{1-p_n}\,\frac{1-q_n}{q_n}\biggr)^{\ft12(T_n-S_n)}
    \biggl(\frac{1-q_n}{1-p_n}\biggr)^{\ft12(|D_{1,n}|-|D_{2,n}|)}\\
    &= Pe^{\ft12\lambda_nS_n}\,Pe^{-\ft12\lambda_nT_n}
    \biggl(\frac{1-q_n}{1-p_n}\biggr)^{\ft12(|D_{1,n}|-|D_{2,n}|)},
  \end{split}
\]
where $\lambda_n:=\log(1-p_n)-\log(p_n)+\log(q_n)-\log(1-q_n)$ and
$(S_n,T_n)$ are distributed binomially, as in the first part of
(\ref{eq:SnTn}). Using the moment-generating function of the
binomial distribution, we conclude that,
\[
  \begin{split}
  &P_{\theta_{0,n}}\biggl(\frac{p_{\theta_n}}
    {p_{\theta_{0,n}}}(X^n)\biggr)^{1/2}
  =\Bigl(1-p_n
      +p_n\Bigl(\frac{1-p_n}{p_n}\,\frac{q_n}{1-q_n}\Bigr)^{1/2}
      \Bigr)^{|D_{1,n}|}\\
  &   \qquad\times\Bigl(1-q_n
      +q_n\Bigl(\frac{p_n}{1-p_n}\,\frac{1-q_n}{q_n}\Bigr)^{1/2}
      \Bigr)^{|D_{2,n}|}
      \biggl(\frac{1-q_n}{1-p_n}\biggr)^{\ft12(|D_{1,n}|-|D_{2,n}|)}\\
  &= \rho(p_n,q_n)^{|D_{1,n}|+|D_{2,n}|},
  \end{split}
\]
which proves the assertion. %
\end{proof}

\subsection{The sizes of \(D_{1,n}\) and \(D_{2,n}\)}

Let \(\theta_n\in \Theta_{\bm m_{n,\ell}}\). There are \(\frac12\sum_{i=1}^{\ell} m_{n,\ell,i}(m_{n,\ell,i}-1)\) pairs \((i,j), i<j\) with \(\theta_{n,i}=\theta_{n,j}\).

 We have \begin{align*}
 \frac12 n(m_{n,\ell,\min}-1)\le \frac12\sum_{i=1}^{\ell} m_{n,\ell,i}(m_{n,\ell,i}-1)\le  \frac12n(m_{n,\ell,\max}-1).
 \end{align*}

\begin{lemma}\label{lem:cardinalityofD1ncupD2n}
Under \cref{ass:assumptiononthesizeoflabelsets},
\[
|D_{1,n}\cup D_{2,n}|\ge \frac12n(m_{n,\ell_0\wedge \ell,\min}-m_{n,\ell_0\vee \ell,\max}).
\]
\end{lemma}
\begin{proof}
As there are at least \(\frac12n(m_{n,\ell_0,\min}-1)\) elements \((i,j), i<j\) for which \(\theta_{0,n,i}=\theta_{0,n,j}\) and at most \(\frac12n(m_{n,\ell,\max}-1)\) elements \((i,j), i<j\) for which \(\theta_{n,i}=\theta_{n,j}\), in case \(\frac12n(m_{n,\ell_0,\min}-1)>\frac12n(m_{n,\ell,\max}-1)\) going from \(\theta_{0,n}\) to \(\theta_n\) there are at least \begin{align*}
&\frac12n(m_{n,\ell_0,\min}-1)-\frac12n(m_{n,\ell,\max}-1)
=  \frac12n(m_{n,\ell_0,\min}-m_{n,\ell,\max}) 
\end{align*}
elements contributed to \(D_{1,n}\).

 In case 
 \(\frac12n(m_{n,\ell,\min}-1)\ge \frac12n(m_{n,\ell_0,\max}-1)\) going from \(\theta_{n}\) to \(\theta_{0,n}\) there are at least \begin{align*}
\frac12n(m_{n,\ell,\min}-1)-\frac12n(m_{n,\ell_0,\max}-1)
=  \frac12n(m_{n,\ell,\min}-m_{n,\ell_0,\max}).
\end{align*}
elements contributed to \(D_{2,n}\). 
It follows from \cref{ass:assumptiononthesizeoflabelsets} that exactly one of these two options occurs. Hence we have that \(D_{1,n}\cup D_{2,n}\) consist of at least \[
\frac12n(m_{n,\ell_0\wedge \ell,\min}-m_{n,\ell_0\vee \ell,\max})
\]
elements.
\end{proof}

\begin{remark}
The lower bound may be obtained. (Let \(1^{k_1}2^{k_2}\ldots \ell^{k_\ell}\) denote the label vector with \(k_1\) consecutive 1's, followed by \(k_2\) consecutive 2's, etc.) Consider an even \(\ell_0\) and let \(\ell=\ell_0/2\) and \(m=2m_0\),  \(\bm m_{n,\ell_0}\) is the \(\ell_0\)-vector \((m_0,\ldots, m_0)\) and \(\bm m_{n,\ell}\) is the \(\ell\)-vector \((m,\ldots, m)\). 
Let \(\theta_{0,n}=(1^{m_0}2^{m_0}\ldots \ell_0^{m_0})\in \Theta_{\bm m_{n,\ell_0}}\) and \(\theta_n=(1^{2m_0}2^{2m_0}\ldots (\ell_0/2)^{2m_0})\in \Theta_{\bm m_{n,\ell}}\). Note that when \(\theta_{0,n,i}=k\), then \(\theta_{n,i}=\ceiling{k/2}\). It follows that \(D_{1,n}\) is an empty set in this case. On the other hand, if \(\theta_{0,n,i}=2k+1\) and \(\theta_{0,n,j}=2k+2\) with \(k\in\set{0,\ldots,\ell_0/2-1}\), then \(\theta_{n,i}=\theta_{n,j}=k+1\). There are \(\frac12\ell_0m_0^2=\frac12 n m_0=\frac12n|m_0-m|\) of such combinations. Hence \(\# D_{2,n}=\frac12n|m_0-m|\) and the lower bound is actually achieved.  
\end{remark}

Obviously, the bound above is useless when \(\theta_n\) and \(\theta_{0,n}\) both belong to the same set \(\Theta_{n,\ell}\). 

The following lemma relates the number of elements in \( D_{1,n}\cup  D_{2,n}\) to the   \(r_{n}\)-distance (\cref{eq:definitionr}) between \(\theta_{0, n}\) and \(\theta_{ n}\). 

\begin{lemma}\label{lem:lowerboundondintermsofr}
	For \(\theta_{0, n},\theta_{ n}\in\Theta_{n,\ell_0}\) with \(r=r_{n}(\theta_{0, n},\theta_{ n})\), we have 
 \[
|D_{1,n}\cup  D_{2,n}| \ge 2r(m_{n,\ell_0,\min}-r)^+, 
\]    
where for a real number \(x, x^+=\max\set{0, x}\).
\end{lemma}
\begin{proof} Let \(\theta_{ 0,n}'\) and \(\theta_n'\) be representations of \(\theta_{0,n} \)
 and \(\theta_n\), respectively, so that \(r_n'(\theta_{0,n}', \theta_n')=r\). There are \(s,t\in\set{1,\ldots,\ell_0}\), \(s\neq t\) so that there are exactly \(r\) vertices with \(\theta_{0,n}'\)-label \(s\) and  \(\theta_n'\)-label \(t\). Clearly, when this happens, \(m_{n,\ell_0,s}\ge r\), and trivially, \(m_{n,\ell_0,s}\ge m_{n,\ell_0,1}\ge m_{n,\ell_0,\min }\). It follows that \(D_{1,n}\) 
has at least \(r(m_{n,\ell_0,s}-r)\ge r(m_{n,\ell_0,\min}-r)^+\) elements. Using  that \(D_{1,n}\) and \(D_{2,n}\) are disjoint and reversing the role of \(\theta_{0, n}\)
 and \(\theta_{n}\) and \(D_{1,n}\) and \(D_{2,n}\) gives \(
 |D_{1,n}\cup  D_{2,n}| \ge 2r(m_{n,\ell_0,\min}-r)^+. 
 \)
\end{proof}

\subsection{The distance \(r_{n}\) and the sets \(V_{n,\ell,k}\)}\label{sec:numberofelementsinVnk}

\begin{lemma}\label{lem:boundonm}
Let \(\ell\le \ell'\). Let \(\theta\in \Theta_{\bm  m_{n ,\ell}}\) and \(\eta\in\Theta_{\bm  {\tilde m}_{n ,\ell'}}\). 
	The distance \(r_n(\theta,\eta)\) is bounded by \(m_{n,\ell,\ell}/2\), which in turn is bounded by \(m_{n,\ell,\max}/2\). 
\end{lemma}
\begin{proof}
	Note that \(\theta\) (\(\eta\), resp.) is 
	characterised by the partition of 
	\(\set{1,\ldots,n}\) in \(\ell\) (\(\ell'\), resp.) sets of 
	sizes \(m_{n,\ell,1},\ldots,m_{n,\ell,\ell}\) (\(\tilde m_{n,\ell',1},\ldots,\tilde m_{n,\ell',\ell'}\), resp.). 
	Vice versa, for such a partition, one can 
	define a labelling, by assigning a label to 
	each set in the partition, and giving each 
	element of the set that label. Define
	\begin{align*}
	\Xi_\theta=\set{\set{i:\theta_i'=a}:a\in\set{1,\ldots,\ell}},
	\end{align*}
	note that \(\Xi_{\theta}\) is independent of 
	the representation, and define \(\Xi_\eta\) 
	similarly. Now note that there are at most 
	\(\ell\) pairs of sets \((A,B),A\in \Xi_\theta,B\in 
	\Xi_\eta\) with 
	\(|A\cap B|> m_{n,\ell,\ell}/2\), because 
	\(\Xi_\theta\)
	consists of \(\ell\) elements and \(\Xi_\eta\) 
	consists of \(\ell'\ge \ell \) elements,
	and if there are more than \(\ell\) such pairs, 
	then there is an \(A\) in  
	\( \Xi_\theta\) and two different \(B_1,B_2\) in \(\Xi_\eta\) such that 
	\(A\cap B_1\) and \(A\cap B_2\) have more 
	than \(m_{n,\ell,\ell}/2\) elements. But 
	\(B_1\) and \(B_2\) are disjoint, so \(A\) has 
	more than \(  m_{n,\ell,\ell}\) elements, 
	which is a contradiction. 
	
	Let \((A_1,B_1),\ldots,(A_k,B_k)\) the pairs 
	with \(|A_i\cap B_i|> m_{n,\ell,\ell}/2\), 
	with \(k\le \ell\)(\(\le \ell'\)). Define a label \(\theta'\in 
	\theta\) by assigning label \(i\) to the sets 
	\(A_i,\) \(i\in\set{1,\ldots,k}\) and the 
	remaining \(\ell-k\) labels  to the remaining 
	sets in the partition \(\Xi_\theta\) and define 
	\(\eta'\in\eta\) similarly by assigning label 
	\(i\) to the sets \(B_i,\) \(i\in\set{1,\ldots,k}\) 
	and the remaining \(\ell'-k\) labels to the 
	remaining sets in \(\Xi_\eta\). It follows that  
	\(r(\theta, \eta) \le r(\theta',\eta')\le  
	m_{n,\ell,\ell}/2\). By the definition of \(m_{n,\ell,\max}\), \(m_{n,\ell,\ell}\le m_{n,\ell,\max}\). 
\end{proof}

Let \(\theta\in\Theta_{\bm m_{n,\ell}}\) and  define \(
V_{n,\ell,k}(\theta)=\set{\eta\in\Theta_{n,\ell}:r_{n}(\theta,\eta)=k}.
\)

\begin{lemma}\label{lem:upperboundVnlk}
	Let \(\theta\in\Theta_{\bm m_{n,\ell}}\), and \(k\le m_{n,\ell,1}\). Then%
	\begin{equation}\label{eq:upperboundfornumberofelementsinVnkintermsofabinomialexponent}
	|V_{n,\ell,k}(\theta)|\le 2^{k\ell(\ell-1)}\binom{n(\ell-1)}{\ell(\ell-1)k}\le \rh{\frac{2en}{k\ell}}^{\ell(\ell-1)k}.
	\end{equation}
\end{lemma}
Note that \cref{eq:upperboundfornumberofelementsinVnkintermsofabinomialexponent} is well defined when \(k\ell \le n\), which is the case when \(k\le m_{n,\ell,1}\). 
\begin{proof}
	When \(\ell=1\), the lemma is trivial. Let \(\ell>1\). Choose a representation \(\theta'\) of \(\theta\) so that exactly \(m_{n,\ell,i}\) labels of \(\theta'\) are \(i\), \(i\in\set{1,\ldots,\ell}\). Let us count the number of ways in which we can change the labelling \(\theta'\) into a label \(\eta'\), while keeping \(r_n(\theta',\eta')\le k\). 
Let \(i, j\in \set{1,\ldots,\ell}, i\neq j\). We can choose \(\binom{m_{n,\ell,i}}{k}\) vertices with \(\theta'\)-label \(i\) and either retain the label \(i\) or give them them the label \(j\). This is possible in \(2^k\) ways. There are \(\ell-1\) labels \(j\neq i\), and there are \(\ell\) labels \(i\). Hence  \(V_{n,\ell,k}\) has at most 
\[
2^{k\ell(\ell-1)}\prod_{i=1}^\ell \binom{m_{n,\ell,i}}{k}^{\ell-1}
\]
elements. 

Using \(\binom ab \binom cd\le \binom{a+c}{b+d}\), the number of elements in \(V_{n,\ell,k}\) is bounded by 
\[
2^{k\ell(\ell-1)}\binom{n(\ell-1)}{\ell(\ell-1)k}.
\]
\Cref{eq:upperboundfornumberofelementsinVnkintermsofabinomialexponent} is in turn  bounded by\[
\rh{\frac{2en}{k\ell}}^{\ell(\ell-1)k}. 
\]
\end{proof}

\subsection{The distance \(m_n\)}

The distance \(m_n\), restricted to \(\Theta_{n,\ell}\) has the following relationship with \(r_{n}\):

\begin{lemma}\label{eq:equivalencemandr}
	For all \(\theta,\eta\in\Theta_{n,\ell}\), \begin{equation}\label{eq:assertionoflemmaequivalencemandr}
	r_{n}(\theta,\eta)\le m(\theta,\eta)\le \ell(\ell-1)r_{n}(\theta,\eta). 
	\end{equation}
\end{lemma}
\begin{proof}
	When \(\ell=1\), \(\Theta_{n,\ell}\) consist of only one element, and \cref{eq:assertionoflemmaequivalencemandr} is trivial. Now suppose \(\ell>1\). 
	Every representation \(\theta'\) of \(\theta\) and \(\eta'\) of \(\eta\) statifies \(r'_{n}(\theta',\eta')\ge r_{n}(\theta,\eta)=:r\). Hence, for fixed representations \(\theta'\) and \(\eta'\), there are at least \(r\) vertices with \(\theta'\) label \(a\) have \(\eta'\) label \(b\), for some \(a, b\in\set{1,\ldots, \ell}, a\neq b\). Hence \(m_n'(\theta',\eta')\ge r\). So \(m_n(\theta,\eta)\ge r\). On the other hand, if \(\theta'\) and \(\eta'\) are representations of \(\theta\) and \(\eta\), respectively, such that  \(r'_{n,\ell}(\theta',\eta')= r_{n,\ell}(\theta,\eta)=:r\), then there are at most \(r\) vertices with \(\theta'\) label 1, that have \(\eta'\) label 2, and at most \(r\) vertices with \(\theta'\) label \(1\) that have \(\eta'\) label 3, etc. As there are \(\ell-1\) labels not equal to 1, and \(\theta'\) has \(\ell\) labels, \(m(\theta,\eta)\le m'(\theta',\eta')\le \ell(\ell-1)r\).
\end{proof}

\section{Auxiliary lemmas}\label{sec:aux}

\begin{lemma}\label{lem:inequalityofanepowerdividedbyoneminusanepower}
	Let \(C\ge 2 \). For all  \(x\ge \sqrt{2/C}\),
	\begin{align*}
	\frac{e^{-Cx}}{1-e^{-x}}\le e^{-Cx/4}.
	\end{align*}
\end{lemma}
\begin{proof}
	Note that \begin{align*}
	1-e^{-x}&=\int_0^{x}e^{-y}dy\\
	&\ge xe^{-x}.
	\end{align*}
	Using this  and the fact that \(C-1\ge C/2\), we see that
	\begin{align*}
	\frac{e^{-Cx}}{1-e^{-x}}\le & \frac{e^{-Cx}}{xe^{-x}}= \frac{e^{-(C-1)x}}{x}\le \frac{e^{-Cx/2}}{x}=\frac{e^{-Cx/4}}{x}e^{-Cx/4}.
	\end{align*}
	As \(x\ge \sqrt{2/C}\)  is equivalent to \(Cx/4\ge x^{-1}/2\) and \(x>0\), we have 
	\begin{align*}
	\frac{e^{-Cx}}{1-e^{-x}}\le & x^{-1}e^{-x^{-1}/2}e^{-Cx/4}.
	\end{align*}
	One verifies that \(f(y)=ye^{-y/2}\) attains its maximum on \(\RR\) at \(y=2\) and \(f(2)=2/e<1\). It now follows that 
	\[\frac{e^{-Cx}}{1-e^{-x}}\le e^{-Cx/4}.\] 
\end{proof}

\begin{lemma}\label{lem:sqrtoneminusxislessthanorequaltooneminusxovertwo}
	For \(x\in[0,1]\), \(\sqrt{1-x}\le 1-x/2\). 
\end{lemma}
\begin{proof}
	Define \(f(x)=\sqrt{1-x}\) and \(g(x)=1-x/2\). Note that \(f(0)=g(0)\) and \(f'(x)=-\frac1{2\sqrt{1-x}}\le -\frac12 = g'(x)\), for all \(x\in[0,1]\). It follows that \(g(x)\ge f(x)\) for all \(x\in[0,1]\). 
\end{proof}

\begin{lemma}\label{lem:oneplusxdivrtothepowerrissmallerthanetothepowerx} 
	For all positive integers \(r\) and real numbers \(x>-r,\) \((1+x/r)^r \le e^x\).
\end{lemma}
\begin{proof}
	Let for \(x>-r\), \(f(x)= r\log(1+x/r)\) and \(g(x)=x\). Then \(f'(x)=\frac1{1+x/r}\) and \(g'(x)=1\). It follows that \(f'(x)\le g'(x)\), when \(x\ge 0\), \(f'(x)>g'(x) \) when \(-n<x<0\) and \(f(0)=g(0)\). It follows that \(f(x)\le g(x)\) for all \(x>-r\). As \(y\to e^y\) is increasing, for all real \(y\), it follows that for all  \(x>-n\), \((1+x/r)^r = e^{f(x)}\le e^{g(x)}=e^x\).
\end{proof}

  \printbibliography
\end{document}